\documentclass[final]{siamltex1213}

\usepackage{amsmath}
\usepackage{amssymb}
\usepackage{amsfonts}
\usepackage{mathrsfs}
\usepackage{algorithm}
\usepackage{algorithmicx}
\usepackage[noend]{algpseudocode}

\newcommand{\RR}{\mathbb{R}}
\newcommand{\Aa}{\mathcal{A}}
\newcommand{\CC}{\mathcal{C}}
\newcommand{\MM}{\mathcal{M}}
\newcommand{\NN}{\mathbb{N}}
\newcommand{\OO}{\mathcal{O}}
\newcommand{\UU}{\mathcal{U}}
\newcommand{\VV}{\mathcal{V}}
\newcommand{\WW}{\mathcal{W}}

\newcommand{\xvec}{\vec{\mathbf{x}}}
\newcommand{\xtil}{\tilde{\mathbf{x}}}
\newcommand{\ntwo}{\hat{\mathfrak{n}}}
\newcommand{\nthree}{\hat{n}}
\newcommand{\first}{1^{\mathrm{st}}}
\newcommand{\second}{2^{\mathrm{nd}}}
\newcommand{\third}{3^{\mathrm{rd}}}

\newcommand{\bull}{\textbullet~}

\newcommand{\ie}{i.e.,~}
\newcommand{\eg}{e.g.,~}
\newcommand{\vs}{vs.~}

\newcommand{\Acc}{\texttt{Accepted}}
\newcommand{\NB}{\texttt{Narrow Band}}
\newcommand{\FA}{\texttt{Far Away}}
\newcommand{\Pp}{\texttt{Pile}}

\newcommand{\Grid}{\textsl{Grid}}
\newcommand{\Norm}{\textsl{Norm}~}
\newcommand{\Orientth}{\textsl{Orient}3~}

\title{A low complexity algorithm for non-monotonically evolving fronts \thanks{Submitted on Friday, September 12th, 2014.} }

\author{Alexandra Tcheng, Jean-Christophe Nave \thanks{Department of Mathematics and Statistics, McGill University, 805 Sherbrooke Street West, Montreal, Quebec H3A 0B9, Canada. Emails: \email{alexandra.tcheng@mail.mcgill.ca}, \email{jcnave@math.mcgill.ca}.}}

\begin{document}
\maketitle

\slugger{sisc}{xxxx}{xx}{x}{x--x}%slugger should be set to mms, siap, sicomp, sicon, sidma, sima, simax, sinum, siopt, sisc, or sirev

\begin{abstract}
A new algorithm is proposed to describe the propagation of fronts advected in the normal direction with prescribed speed function $F$. 
The assumptions on $F$ are that it does not depend on the front itself, but can depend on space and time. Moreover, it can vanish and change sign. To solve this problem the Level-Set Method [Osher, Sethian; 1988] is widely used, and the Generalized Fast Marching Method [Carlini et al.\mbox{}; 2008] has recently been introduced. The novelty of our method is that its overall computational complexity is predicted to be comparable to that of the Fast Marching Method [Sethian; 1996], [Vladimirsky; 2006] in most instances. This latter algorithm is $\mathcal{O}(N^{n}\log N^{n})$ if the computational domain comprises $N^{n}$ points. Our strategy is to use it in regions where the speed is bounded away from zero -- and switch to a different formalism when $F \approx 0$. To this end, a collection of so-called \emph{sideways} partial differential equations is introduced. Their solutions locally describe the evolving front and depend on both space and time. The well-posedness of those equations, as well as their geometric properties are adressed. We then propose a convergent and stable discretization of those PDEs. Those alternative representations are used to augment the standard Fast Marching Method. The resulting algorithm is presented together with a thorough discussion of its features. The accuracy of the scheme is tested when $F$ depends on both space and time. Each example yields an $\mathcal{O}(1/N)$ global truncation error. 
We conclude with a discussion of the advantages and limitations of our method. %, as well as potential extensions.
\end{abstract}



\pagestyle{myheadings}
\thispagestyle{plain}
\markboth{A.TCHENG, J.-C.NAVE}{A low complexity algorithm for evolving fronts}

\section{Introduction}
\label{sec:Introduction}

%{\magenta To finish biblio work: Look at Osher-Sethian's intro, FedkiwOsher's book, and at least one recent article.}
%- Nothing on medical imaging, movie animation!

The design of robust numerical schemes describing front propagation has been a subject of active research for several decades. The need for such schemes is felt across many areas of applied sciences: geometric optics \cite{OsherTsai}, optimal control \cite{FalconeMin,TakeiTsai}, lithography \cite{AdalSethian1,AdalSethian2,AdalSethian3}, shape recognition \cite{ShapeRecog1,ShapeFromShading}, dendritic growth \cite{Dendritic1,Dendritic2}, gas and fluid dynamics \cite{TriplePoint,GasDynamics,LevelSetFluids}, combustion \cite{Combustion}, etc. Depending on the problem at hand, various issues may arise. Consider the following two interface propagation phenomena: A fire propagating through a forest, and a large evolving population of bacteria in a Petri dish. In either case, space can be divided into distinct regions: burnt \vs unburnt, and populated \vs unpopulated. The boundaries between those regions form fronts that evolve in time. Those examples differ from one another in that a fire front can only propagate \emph{monotonically}, whereas bacteria may advance or recede, depending on the stimuli present in their environment. This distinction led to different approaches when modelling those evolutions. Monotone propagation can be recast into a `static' problem, as opposed to non-monotone evolution, which is instrinsically time-dependent. As a result, efficient single-pass algorithms for monotone propagation have been developed. In contrast, accurate algorithms for non-monotonically evolving fronts require a larger number of computations. In this paper, we propose a model that reconciles the advantages of previous methods -- We accurately describe non-monotone front evolution with an algorithm that performs a low number of operations.

% Consider for example a population of bacteria living on a Petri dish. Such organisms respond to chemical stimuli by moving in the direction of greatest increase (or decrease) of the stimulus \textbf{cite} -- i.e., if nutrients are added to their environment, the bacteria will move along the resulting food gradient. Provided the number of micro-organisms is large enough, the Petri dish can be divided into distinct regions: The ones occupied by the bacteria, and the non-invaded ones. The boundary between those regions form fronts that evolve in time. By controlling the stimuli given to the population, such as food or heat, biologists directly influence the motion of the fronts. This purely motivational example is the kind of situation that we wish to describe. 

%About re-initialization in the LSM: \cite{ReInit1} \cite{SussmanFatemi} \cite{Reinit2} \cite{TsaiRedistancing}
%Narrow Band LSM: \cite{AdalSethian} \cite{ReInit1}
%{\red general ref on the LSM: \cite{sethian1999level}, \cite{OsherFedkiw}.}
%{\red Evans \& Spruck work: \cite{EvansSpruck1} \cite{EvansSpruck2} \cite{EvansSpruck3} \cite{EvansSpruck4}
One of the early means of accurately propagating fronts was to use the Level-Set Method (LSM) \cite{OsherSethian}. This implicit approach embeds the front as the zero-level-set of an auxiliary function $\phi$. In the above example, $\phi$ could be negative in regions occupied by bacteria, and positive in other regions. Each contour of this level-set function is then evolved under the given speed function $F$, which guarantees that the front itself moves properly. The robustness and simplicity of the first order discretization of this problem made it popular. Additionally, this approach can handle a very wide class of speed functions, including those that change sign. However, describing the evolution of an $(n-1)$-dimensional front in $\RR^{n}$ requires solving for a function of $n+1$ variables, since $\phi$ depends on space as well as time. Moreover, in order for the solution to remain accurate, it is often desirable to enforce the signed distance property $|\nabla \phi | \approx 1$ in a neighbourhood of the front. There exists a vast literature on lowering the computational complexity of the LSM, cf.\mbox{} \cite{AdalSethian,ReInit1,sethian1999level,OsherFedkiw}, and on maintaining the accuracy of the solution, cf.\mbox{}  \cite{TsaiRedistancing,chopp2001some,ReInit1,SussmanFatemi,Reinit2}. Nevertheless, those features are incorporated at the expense of the simplicity and the efficiency of the original LSM. 

The Fast Marching Method (FMM) \cite{Sethian,Tsitsi} constitutes the second significant advance in the field. This approach requires the speed function to be bounded away from zero, and to be only space-dependent. %In our toy model, this would occur if food was only placed along the rim of the dish. 
Under those conditions, the FMM builds the `first arrival time' function $\psi$ such that to every point $\xvec$ in space is associated the value $t=\psi(\xvec)$ at which the front reaches $\xvec$, cf.\mbox{} \cite{Sethian,SethianFMM,SethianBookVariational,sethian1999level}. In the context of fire propagation, $\psi$ records the time at which the parcel of land burnt. The use of a Dijkstra-like data structure \cite{Dijkstra} renders this scheme very efficient. A variant of this algorithm known as the Fast Sweeping Method runs in $\mathcal{O}(N^{n})$ complexity \cite{zhao2005fast} when the computational domain comprises $N^{n}$ points. Recently, Falcone et al. \cite{GFMM} proposed a Generalized FMM (GFMM) that is able to handle vanishing speeds. This algorithm is supported by theoretical results on its convergence in the class of viscosity solutions. The examples presented are found to accurately propagate the fronts subject to a wide range of speed functions. However, when $F$ depends on time, the GFMM no longer makes use of a Dijkstra-like data structure. Its overall complexity is expected to revert to that of the LSM in such instances. 

In the light of this previous work, it is desirable to design an algorithm able to handle speed functions that change sign, while retaining the efficiency of the FMM. This is the main purpose of this article.
Note that if $F$ changes sign, a point $\xvec$ in space may be reached by the front several times. This implies that the arrival time can no longer be described as a function depending solely on space. 
However, it is still possible to locally describe it as the graph of a function. Consider the set $\MM := \{ (\xvec,t) : \xvec $ belongs to the front at time $t \}$. The set $\MM$ consists of the surface traced out by the fronts as they evolve through space and time. If $\MM$ embeds as a $C^{k}$-manifold of dimension $n$ in $\RR^{n} \times (0,T)$, then by definition, each point $(\xvec,t)\in \MM$ belongs to a neighbourhood that is locally the image of a $C^{k}$-function of $n$ variables. The fact that under mild assumptions $\MM$ is a compact subset of $\RR^{n} \times (0,T)$ guarantees that we only need a finite number of neighbourhoods to cover $\MM$, or equivalently, a finite number of functions to parametrize $\MM$. The images of those functions -- which possibly depend on time as well as space -- provide local representations of the set $\MM$. Our approach makes use of those other representations whenever the purely spatial one is not available -- \eg when $n=2$ and $\MM$ cannot be locally described by the standard first arrival time function $\{ t = \psi(x,y) \}$, we may describe it as $\{ x = \tilde{\psi}(y,t) \}$ or $\{ y = \bar{\psi}(x,t) \}$.
To this end, we introduce \emph{sideways} PDEs solved by those $C^{k}$-functions. We illustrate in detail how they relate to previous work, argue that they are well-posed, and show that their solution does provide a local description of $\MM$. Moreover, we provide a scheme to discretize them, prove that it converges to the correct viscosity solution, and show that it is stable.   

In practice, the proposed algorithm amounts to augmenting the FMM to be able to describe $\MM$ near those points $(\xvec,t)$ where $F(\xvec,t)=0$. The fact that different representations are used to build different parts of $\MM$ implies that those pieces need to be woven together along their overlapping parts, to form a single codimension one subset of $\RR^{n} \times (0,T)$. This is done by storing the $(n+1)$-dimensional normal associated to each point and by using interpolation.
To illustrate the overall method, examples are presented where an $\mathcal{O}(1/N)$ global truncation error is achieved. Those tests all feature speed functions that vanish, and possibly depend on time. %Our second example investigates the robustness of the scheme by evolving a circle under a speed function $F(t) \propto \sin(t)$. We also look at the case where a curve initially expands and later contracts until it collapses to a point. It does so while featuring a time- and space-dependent shock. 

Since the algorithm always approximates a function of $n$-variables, the dimensionality of the problem is never raised, unlike what happens in the LSM. As a result, the computational complexity is expected to be comparable to that of the FMM.

%%%%%%%%%%%%%%%%%%%%%%%%%%%%%%%%%%%
%%%%%%%%%%%%%%%%%%%%%%%%%%%%%%%%%%%
%%%%%%%%%%%%%%%%%%%%%%%%%%%%%%%%%%%

\paragraph{Outline of the article}
This paper is organized as follows. We state the problem we are addressing in \S \ref{sec:Preliminaries}. We also present the LSM and the FMM, before providing a simple example to motivate our method. The case where $F$ is bounded away from zero and depends on time is addressed in \S \ref{subsec:tFMM}. The \emph{sideways} PDEs we use in regions where $F \approx 0$ are introduced in \S \ref{subsec:FrontFunction}. A discussion of their properties is provided along with a convergent and stable scheme to discretize them. %The ideas from previous sections are bridged into the algorithms that follow in \S \ref{sec:Weaving} -- 
We explain how the different formalisms can be woven into a single method in \S \ref{sec:Weaving}. The pseudo-codes are given and discussed in \S \ref{sec:AlgoDiscussion}. We predict the complexity and accuracy of the overall method in \S \ref{sec:Complexity} and \S \ref{sec:Accuracy}. Four examples are then covered in details in \S \ref{sec:Accuracy}. Those assess the global behaviour and the accuracy of the scheme. The advantages and weaknesses of our approach are discussed in \S \ref{sec:Discussion}, where an additional example is covered to address the limitations of the method. We conclude in \S \ref{sec:Conclusions}. 

%%%%%%%%%%%%%%%%%%%%%%%%%%%%%%%%%%%%
%%%%%%%%%%%%%%%%%%%%%%%%%%%%%%%%%%%%
%%%%%%%%%%%%%%%%%%%%%%%%%%%%%%%%%%%%

\section{Preliminaries}
\label{sec:Preliminaries}

\subsection{Problem statement}
\label{subsec:Statement}

Let a subset $\CC_{0} \subset \RR^{n}$ be closed with no boundary. Assume it is an orientable manifold of codimension one, with a well defined unique outer normal $\ntwo_{0} (\xvec)$. Suppose $\CC_{0}$ is advected in time, and denote the resulting subset of $\RR^{n}$ at time $t$ by $\CC_{t}$. We want to describe $\CC_{t}$ for $0<t<T$ in the case where each point $\xvec \in \CC_{t}$ is advected under the velocity 
\begin{eqnarray}
\vec{v} = \vec{v}(\xvec,t) = F(\xvec,t) \ntwo(\xvec,t)
\end{eqnarray}
i.e., with the prescribed speed function $F=F(\xvec,t)$, in the direction of the outward normal to $\CC_{t}$, $\ntwo = \ntwo(\xvec,t)$. 

%%%%%%%%%%%%%%%%%%%%%%%%%%%%%%%%%%%
%%%%%%%%%%%%%%%%%%%%%%%%%%%%%%%%%%%

\subsection{Assumptions}
\label{subsec:Assumptions}

In addition to the assumptions already stated, in the rest of this paper we assume that the following hold. The initial set $\mathcal{C}_{0}$ is known exactly, and is assumed to be $C^{2}$ in the sense that if it is given as the image of a map, \eg $\vec{\gamma}: S^{n-1} \longrightarrow \RR^{n}$, then $\vec{\gamma} \in C^{2}(S^{n-1})$. %We further assume that its representation as the zero-level-set of some function is available \eg $\phi_{0}(\vec{\mathbf{x}})=r(\vec{\mathbf{x}})-r_{0}$. 
The speed $F=F(\vec{\mathbf{x}},t)$ is known exactly for all $(\vec{\mathbf{x}},t)$. Unless otherwise specified, it is allowed to vanish and change sign. It does not depend on the curve itself, or any of its derivatives. For simplicity, we also make the following strong assumption: the map $F : \RR^{n} \times (0,T) \longrightarrow \RR $ is analytic. In particular, this implies that the subset defined as $\mathcal{F} : = \{ (\xvec,t) : F(\xvec,t)=0 \}$ is closed and has codimension one in $\RR^{n}\times [0,T]$. %The fact that there exist constants $0<K_{1}<K_{2}<\infty$ such that $K_{1} \leq |\nabla F| \leq K_{2}$ is often used. 
We let $K$ be the Lipschitz constant of $F$. 
Together, those assumptions guarantee that for any given $t\in (0,T)$, there exists a well defined normal $\ntwo = \ntwo(\xvec,t)$ almost everywhere along $\CC_{t}$.

%%%%%%%%%%%%%%%%%%%%%%%%%%%%%%%%%%%
%%%%%%%%%%%%%%%%%%%%%%%%%%%%%%%%%%%

\subsection{Previous Work}
\label{subsec:PreviousWork}

For completeness we briefly go over two of the methods mentioned in the introduction. Considering that the set $\RR^{n} \setminus \CC_{t}$ consists of two connected components, we define $\Aa_{t}$ to be the bounded one. % We will therefore refer to the unbounded one as $\Aa^{c}_{t} \setminus \CC_{t}$. %See the right part of Figure \ref{fig:NotationPartI} for an illustration. \footnote{{\red Maybe I can remove that figure.}}

% Here, I'm citing: 
%About re-initialization in the LSM: \cite{ReInit1} \cite{SussmanFatemi} \cite{Reinit2} \cite{TsaiRedistancing}
%Narrow Band LSM: \cite{AdalSethian} \cite{ReInit1}
%{\red general ref on the LSM: \cite{sethian1999level}, \cite{OsherFedkiw}.}
%{\red Evans \& Spruck work: \cite{EvansSpruck1} \cite{EvansSpruck2} \cite{EvansSpruck3} \cite{EvansSpruck4}
\subsubsection{The Level-Set Method}
\label{subsubsec:LSM}

This approach was introduced by Osher \& Sethian in \cite{OsherSethian}. Their idea is to embed the curve $\CC_{t}$ as the zero-level-set of a function $\phi: \RR^{n} \times [0,T] \rightarrow \RR$, i.e., $\CC_{t} = \{ \xvec : \phi(\xvec,t) = 0 \}$. In this setting, the outward normal $\ntwo(\xvec,t)$ is $\frac{\nabla \phi}{|\nabla \phi|}$. The Level-Set Equation is derived from linear advection $\phi_{t} + \vec{v}(\xvec,t) \cdot \nabla \phi = 0$ to yield the following Initial Value Problem (IVP):
\begin{eqnarray}
\label{eq:LSE}
\left\{ \begin{array}{rcll} 
\phi_{t}+F|\nabla\phi| &=& 0 & \quad \mathrm{on}~ \RR^{n} \times (0,T) \\
\phi(\xvec,0) &=& \phi_{0}(\xvec) & \quad \mathrm{on}~ \RR^{n} \times \{ 0 \}
\end{array} \right.
\end{eqnarray}
where $\phi_{0}(\xvec)$ is such that $\{ \xvec : \phi_{0}(\xvec) =0 \} = \CC_{0}$. This method enjoys many desirable properties that have been studied in a variety of contexts \cite{EvansSpruck1,EvansSpruck2,EvansSpruck3,EvansSpruck4,OsherFedkiw,sethian1999level}. One of the most prominent is that topological changes are accurately handled, and do not require special treatment. In \cite{OsherSethian}, the authors propose various discretizations of this evolution on a spatial domain that comprises $N^{n}$ points. The resulting method has complexity $\mathcal{O}(N^{n})$ at each time step, due to the fact that all the contours of the level-set function are advected. To lower this high computational cost, it is possible to work only within a neighbourhood of the zero-level-set: This yields the Narrow Band LSM \cite{AdalSethian}. To be able to render the curve $\CC_{t}$ accurately, it is desirable to preserve the signed distance property $|\nabla \phi| \approx 1$. To this end, the reinitialization method has been studied extensively \cite{ReInit1,Reinit2,SussmanFatemi,TsaiRedistancing}. Early versions of this method tend to displace the zero-level-set, yielding inaccuracies in the final $\CC_{T}$. Moreover, they usually involve a large number of computations. 

%%%%%%%%%%%%%%%%%%%%%%%%%%%%%%%%%%%

%{\red On the Eikonal equation and its relation to the minimum-time problem: \cite{FalconeMin}.} 
\subsubsection{The Fast Marching Method}
\label{subsubsec:FMM}
The Fast Marching Method was independently proposed by Sethian \cite{Sethian} \& Tsitsiklis \cite{Tsitsi}. Strongly rooted in control theory, it requires that $F=F(\xvec) \geq \delta >0$ on $\RR^{n}$. Under those conditions, the FMM solves the following Eikonal equation, whose unknown is the time $\psi:\RR^{n} \mapsto \RR$ at which each point is reached by the curve 
\begin{eqnarray}
\label{eq:EikonalNoTime}
\left\{ \begin{array}{rcll} 
|\nabla\psi| &=& \frac{1}{F} & \quad \mathrm{on}~ \Aa^{c}_{0} \setminus \CC_{0} \\ 
\psi(\xvec) &=& 0 & \quad \mathrm{on}~ \CC_{0}
\end{array} \right.
\end{eqnarray} 
The FMM makes use of a Narrow Band to advance the front in a manner that enforces the characteristic structure of the PDE into the solution. See \cite{Sethian,SethianFMM,SethianBookVariational,sethian1999level} and \cite{FalconeMin} for details. Recent improvements of this method include on the one hand the work of Zhao \cite{zhao2005fast}, who further lowered the complexity of the algorithm to develop the Fast Sweeping Method. On the other hand, Vladimirsky relaxed the restrictions on the speed by allowing it to be time-dependent. We discuss this latter method in \S \ref{subsec:tFMM}.

%%%%%%%%%%%%%%%%%%%%%%%%%%%%%%%%%%%
%%%%%%%%%%%%%%%%%%%%%%%%%%%%%%%%%%%

\subsection{Motivation}
\label{subsec:Motivation}

We first present a simple example to motivate the need for an augmented FMM. Consider the initial curve $\mathcal{C}_{0} = \{ \vec{\mathbf{x}} : x^{2}+y^{2} = r_{0}^{2} \} \subset \RR^{2}$ and the time-dependent speed $F(t) = 1-ct$, where $c$ and $r_{0}$ are positive constants. Let $\phi_{0}(\vec{\mathbf{x}})$ be the signed distance function $ \phi_{0}(\vec{\mathbf{x}}) =  \sqrt{x^{2}+y^{2}}- r_{0} = : r(\vec{\mathbf{x}}) - r_{0}$. The exact solution to the IVP (\ref{eq:LSE}) is then $\phi(x,y,t) = r(\vec{\mathbf{x}}) -\left( r_{0} - \left( c \, t^{2}/2-t\right) \right)$.
The evolution of the curve can be formally split into two parts: \textbf{(1)} For $t \in [0,\frac{1}{c}]$, the circle expands until it reaches the maximal radius $R = r_{0}+\frac{1}{2c}$. \textbf{(2)}  For $t \in (\frac{1}{c}, T]$, where $T=-\frac{1}{c}\left( 1 - \sqrt{1+2cr_{0}}\right)$, the circle contracts until it collapses to the point $(0,0)$ at time $T$.

\begin{figure}
		%\centering
			% Ratio height/width = .73
			%\includegraphics[width=\textwidth,natheight=864,natwidth=1185]{RugbyDraftLowRes.png}
			\includegraphics[width=\textwidth]{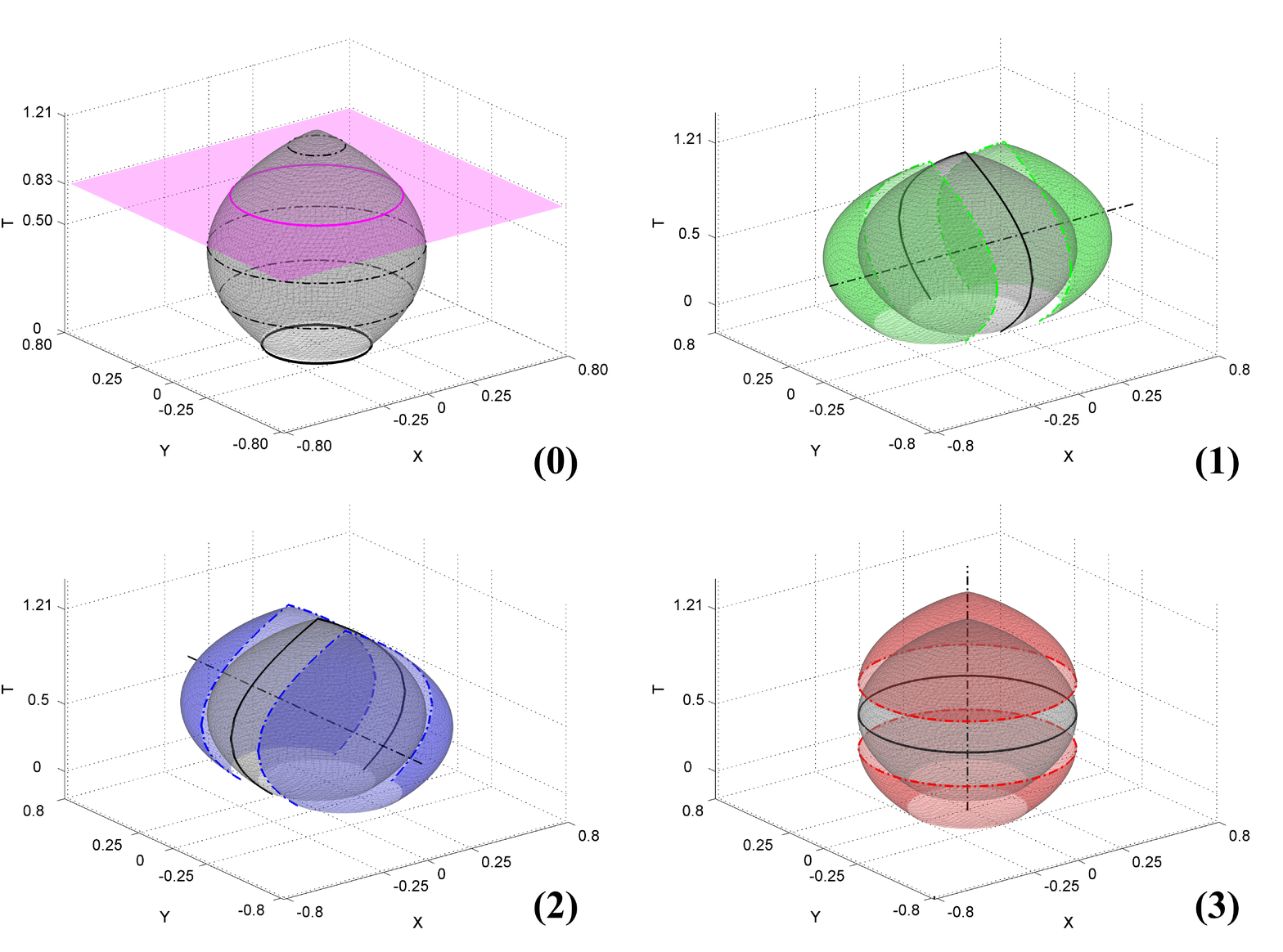}
		\caption{\textsc{Chart decomposition of $\MM$} when $c=2$ and $r_{0}=0.25$. The circle collapses to $(0,0)$ at time $T\approx 1.21$. \textbf{(0)} The manifold, sliced by the plane $t=0.83$ to yield the magenta curve $\CC_{0.83}$.  Three other typical curves $\CC_{t}$ are featured with dashed lines. $\CC_{0}$ appears as a thick plain line. \textbf{(1)} $\MM$ along with $\WW_{1,-}$ and $\WW_{1,+}$ appearing in green. \textbf{(2)} $\MM$ along with $\WW_{2,-}$ and $\WW_{2,+}$ appearing in blue. \textbf{(3)} $\MM$ along with $\WW_{3,-}$ and $\WW_{3,+}$ appearing in red. }
	\label{fig:RugbyDraft}
\end{figure} 

Consider the following atlas $\mathscr{A}$ to describe the resulting $C^{0}$-manifold $\MM$ featured on Figure \ref{fig:RugbyDraft}. Let $\UU:=\RR \times [0,T]$. Then $\mathscr{A}=\cup^{3}_{i=1} \{ ( \psi_{i,\pm}, \WW_{i,\pm}) \}$ where the real-valued functions $ \psi_{i,\pm}$ are defined as: 
\begin{eqnarray}
\begin{array}{ll}
\psi_{1,- }: \UU \longrightarrow [- R,0] & \qquad
\psi_{1,+}: \UU \longrightarrow [0,R] \\
\psi_{2,- }: \UU \longrightarrow [- R,0] & \qquad
\psi_{2,+}: \UU \longrightarrow [0,R]  \\
\psi_{3,- }: \RR^{2} \longrightarrow [0,\frac{1}{c}) & \qquad
\psi_{3,+}: \RR^{2} \longrightarrow (\frac{1}{c},T] 
\end{array}
\end{eqnarray} 
and
\begin{eqnarray}
\psi_{1,\pm}(y,t) &=& \pm \sqrt{\left( r_{0}- c\, t^{2}/2+t \right)^{2} - y^{2}} \\
\psi_{2,\pm}(x,t) &=& \pm \sqrt{\left( r_{0}- c\, t^{2}/2+t \right)^{2} - x^{2}} \\
\psi_{3,\pm}(x,y) &=& \frac{1}{c}\left( 1 \pm \sqrt{1-2c(r(\vec{\mathbf{x}}) -r_{0})}\right) 
\end{eqnarray} 
We also define the sets $\WW_{i,\pm}$ as the real part of the image of the functions $\psi_{i,\pm}$. Those sets are featured on Figure \ref{fig:RugbyDraft}. The functions $\psi_{3,\pm}$ can be verified to be the unique classical solutions to: 
\begin{eqnarray}
\left\{ \begin{array}{rcll} 
|\nabla \psi_{3,-}(\vec{\mathbf{x}})| &=& \frac{1}{F(\psi_{3,-}(\vec{\mathbf{x}}))} & \quad \mathrm{on}~ \UU_{3,-}  \\
\psi_{3,-}(\vec{\mathbf{x}}) &=& 0 & \quad \mathrm{on}~ \mathcal{C}_{0}
\end{array} \right.
\end{eqnarray} 
\begin{eqnarray}
\left\{ \begin{array}{rcll} 
|\nabla \psi_{3,+}(\vec{\mathbf{x}})| &=& - \frac{1}{F(\psi_{3,+}(\vec{\mathbf{x}}))} & \quad \mathrm{on}~ \UU_{3,+} \\
\psi_{3,+}(\vec{\mathbf{x}}) &=& \frac{1}{c} & \quad \mathrm{on}~ \mathcal{C}_{1/c}
\end{array} \right.
\end{eqnarray} 
where $\UU_{3,-} = \{ \xvec :  r_{0}< r(\xvec) < R \}$ and $\UU_{3,+} = \{ \xvec :  0 \leq r(\xvec) < R \}$. Together, the graphs of $\psi_{3,-}$ and $\psi_{3,+}$ describe all of $\MM$ but the circle of radius $R$ reached at time $t=\frac{1}{c}$. On the other hand this circle lies in the union of the images of $\psi_{1,\pm}$ and $\psi_{2,\pm}$. Those functions are the unique classical solutions to
\begin{eqnarray}
\left\{ \begin{array}{cl} 
\mp (\psi_{1,\pm})_{t} + F(t) \sqrt{1+(\psi_{1,\pm})^{2}_{y}} = 0 & ~ \mathrm{on}~ \RR \times (0,T] \\
\psi_{1,\pm}(y,0) = \pm \sqrt{r^{2}_{0}-y^{2}} & ~ \mathrm{on}~ \RR \times \{ 0 \}
\end{array} \right.
\end{eqnarray} 
\begin{eqnarray}
\left\{ \begin{array}{cl} 
\mp (\psi_{2,\pm})_{t} + F(t) \sqrt{(\psi_{2,\pm})^{2}_{x}+1} = 0 & ~ \mathrm{on}~ \RR \times (0,T] \\
\psi_{2,\pm}(x,0) = \pm \sqrt{r^{2}_{0}-x^{2}} & ~ \mathrm{on}~  \RR \times \{ 0 \}
\end{array} \right.
\end{eqnarray} 

This suggests the following procedure to build $\mathcal{M}$: \textbf{(1)} First, solve for $\psi_{3,-}$. \textbf{(Inter.)} Then solve for $\psi_{1,\pm}$ and $\psi_{2,\pm}$ restricted to $[-R,R] \times [\frac{1}{c}-\epsilon,\frac{1}{c}+\epsilon]$ for some $\epsilon>0$. \textbf{(2)} Finally, solve for $\psi_{3,+}$.

Some questions immediately come to mind. Criteria to decide when to move from (1) to the intermediate step must be chosen. Similarly, knowing which equation to solve within the intermediate step is a concern. The practical aspects of how a code reconciles the results of those steps need to be addressed carefully. We discuss all of these issues, and, as a result, turn the above formal idea into an efficient algorithm that constructs $\mathcal{M}$.

%%%%%%%%%%%%%%%%%%%%%%%%%%%%%%%%%%%
%%%%%%%%%%%%%%%%%%%%%%%%%%%%%%%%%%%
%\newpage

\subsection{Notation}
\label{subsec:NotationI}

To lighten the notation, we will now work in the setting where $n=2$. All the results discussed extend to arbitrary $n$. 

\paragraph{Continuous setting}

We use the letter $\psi$ to denote functions whose image locally describes $\MM$. Suppose $\psi : \UU \mapsto \RR$ with $\psi : (y,t) \mapsto \psi(y,t)=x$. We introduce the following subsets of $\RR^{2}$:
\begin{eqnarray}
\Gamma_{t} &:=& \{ (x,y) \in \mathbb{R}^{2} : \psi(y,t)=x, (y,t) \in \UU \} 
\end{eqnarray}
%Those will be useful when discussing local representations of the curve $\CC_{t}$.
See Figure \ref{fig:NotationPartI} for an illustration. We distinguish between $\ntwo(\xvec,t)$ the two-dimensional outward normal to $\CC_{t}$ at $\xvec$; and $\nthree(\xvec,t)$ the three-dimensional outward normal to $\MM$ at $(\xvec,t)$.

\paragraph{Discrete setting}

The spatial grids have fixed meshsize $\Delta x = \Delta y =: h$. We use 
\begin{eqnarray}
x_{i}=i \cdot h
\quad 
y_{j}=j \cdot h
\quad 
t^{k}=k \cdot \Delta t
\qquad
(i,j,k) \in \mathbb{Z}\times \mathbb{Z} \times \{ \mathbb{N}\cup \{0 \} \}
\end{eqnarray}
to denote discrete values of space and time. We usually make no distinction between the continuous functions $\psi$ and their discrete approximations, except in \S \ref{subsec:FrontFunction}. We will be using indices consistently, so that $\psi_{ij}$ can be understood as $\psi(x_{i},y_{j})$ and $\psi^{k}_{i}$ as $\psi(x_{i},t^{k})$. Nevertheless, we will explicitly mention which representation is used. If a point $p$ belongs to $\mathcal{M}$, then it may be described by one or more of the following three expressions: 
\begin{eqnarray}
p^{k}_{j} = (\psi^{k}_{j},y_{j},t^{k})
\qquad
p^{k}_{i} = (x_{i},\psi^{k}_{i},t^{k})
\qquad
p_{ij} = (x_{i},y_{j},\psi_{ij})
\end{eqnarray}

\begin{figure}
		\centering
			\includegraphics[width=0.5\textwidth]{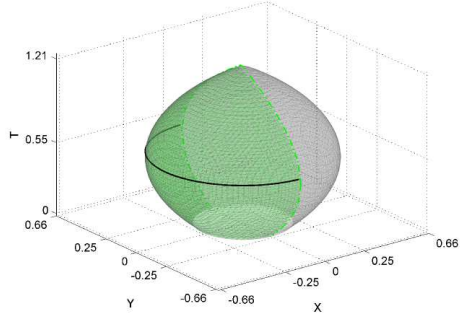}
		\caption{The subset $\Gamma_{0.55}$ associated with $\psi_{1,-}$ from \S \ref{subsec:Motivation} appears as a black plain line.}
	\label{fig:NotationPartI}
	\end{figure}

%%%%%%%%%%%%%%%%%%%%%%%%%%%%%%%%%%%
%%%%%%%%%%%%%%%%%%%%%%%%%%%%%%%%%%%
%%%%%%%%%%%%%%%%%%%%%%%%%%%%%%%%%%%
%\newpage

\section{A FMM for time-dependent speeds: The $t$-FMM}
\label{subsec:tFMM}

We first address the problem stated in \S \ref{subsec:Statement} under the following restriction: 
\begin{eqnarray}
F=F(\xvec,t)\geq \delta >0 \qquad \quad \forall ~ (\xvec,t) \in \RR^{2} \times [0, T]
\end{eqnarray}
Allowing the speed to depend on time yields a non-autonomous control problem. In \cite{Vlad}, the author studies this min-time-from-the-boundary problem in the context of anisotropic front propagation. In our context, the main result of \cite{Vlad} may be formulated as follows: The value function $\psi$ for this control problem satisfies the following Hamilton-Jacobi-Bellman equation: 
\begin{eqnarray}
H(\nabla \psi, \psi, \xvec) := || \nabla \psi(\xvec) || F\left( \xvec, \psi(\xvec) \right) = 1 
\end{eqnarray}
The implementation of the resulting boundary-value problem: 
\begin{eqnarray}
\left\{ 
\begin{array}{rcll}
|| \nabla \psi(\xvec) || &=& \frac{1}{ F\left( \xvec , \psi(\xvec) \right) } \leq \frac{1}{\delta} & \qquad \mathrm{on}~ \mathcal{A}^{c}_{0} \setminus \CC_{0}  \\
\psi(\xvec) &=& 0 & \qquad \mathrm{on}~ \CC_{0}
\end{array} 
\right.
\end{eqnarray}
closely mimicks that of the classical FMM. The only step that requires modifications is the one where a tentative value is assigned to each point in the Narrow Band. Following \cite{Vlad} this step is adjusted as follows. Let $\xvec_{ij}=(x_{i},y_{j})$. Without loss of generality, assume that $\xvec_{i-1,j}$ and $\xvec_{i,j+1}$ are Accepted neighbours of $\xvec_{ij}$. Consider a straight line lying in Quadrant II and ending at $\xvec_{ij}$, and suppose it intersects the line joining $\xvec_{i-1,j}$ and $\xvec_{i,j+1}$ at the point $\xtil$. See Figure \ref{fig:Quadrant}. Then: $\xtil = \xi \xvec_{i-1,j}+(1-\xi)\xvec_{i,j+1}$ for some $\xi \in [0,1]$. Letting $\vec{v} = \xvec_{ij}-\xtil$, we get $|\vec{v}| = \sqrt{\xi^{2}+(1-\xi)^{2}}~h$. Associate the following value to Quadrant II:
\begin{eqnarray}
\label{eq:VladMinimization}
\psi_{\mathrm{II}} = \min_{\xi\in [0,1]} \left\{ \psi(\xtil) + \sqrt{\xi^{2}+(1-\xi)^{2}}~ \frac{~h}{F(\xvec_{ij},\psi(\xtil))} \right\}
\end{eqnarray}
Proceeding similarly in the other quadrants yields the values $\psi_{\mathrm{I}}$, $\psi_{\mathrm{III}}$ and $\psi_{\mathrm{IV}}$. The tentative value assigned to $\psi_{ij}$ is then $\psi_{ij} = \min \{ \psi_{\mathrm{I}}, ~ \psi_{\mathrm{II}},~ \psi_{\mathrm{III}},~ \psi_{\mathrm{IV}} \}$.
Note that in two dimensions the minimization problem (\ref{eq:VladMinimization}) may be solved using a direct method; see Appendix \ref{app:tFMM}. 
This method converges to the correct viscosity solution, and is globally $\first$ order \cite{OUM1,OUM2,Vlad}. Its complexity is $\mathcal{O}(N^{n}\log N^{n})$. %Moreover, it is \textit{a priori} applicable in any number of dimensions $n$. 
In subsequent sections of this paper, we will refer to this modified FMM as the `$t$-FMM'. The results presented in this section yield Algorithm \ref{AlgotFMM} given in \S \ref{sec:AlgoDiscussion}.
Finally, %if $F=F(\xvec,t)\leq - \delta <0$, we get $|| \nabla \psi(\xvec) || F\left( \xvec, \psi(\xvec) \right) = -1$. So that 
in the general case $|F| \geq \delta > 0$, the PDE we wish to solve is $|| \nabla \psi(\xvec) || \, |F\left( \xvec, \psi(\xvec) \right)| = 1$.

\begin{figure}
		\centering
			\includegraphics[width=0.7\textwidth]{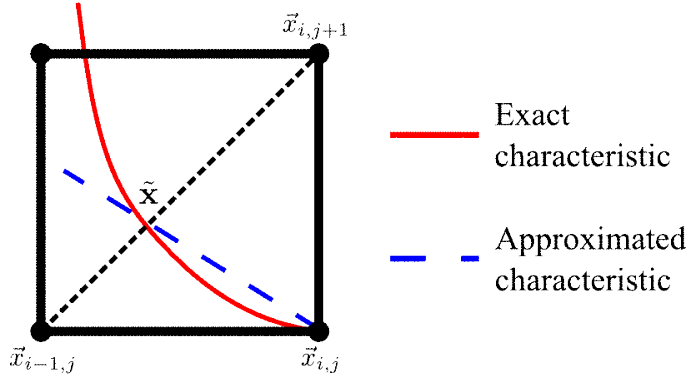}
		\caption{ If the characteristic comes from Quadrant II.}
	\label{fig:Quadrant}
	\end{figure} 

%%%%%%%%%%%%%%%%%%%%%%%%%%%%%%%%%%%
%%%%%%%%%%%%%%%%%%%%%%%%%%%%%%%%%%%
%%%%%%%%%%%%%%%%%%%%%%%%%%%%%%%%%%%
%\newpage 

\section{A local description of the evolving front: The sideways representation}
\label{subsec:FrontFunction}
An option to study the evolution of propagating curves or surfaces is to represent the front as a function that depends on time, \eg $y=Y(x,t)$ \cite{SethianFMM}. Although successful at describing the evolution locally, this approach fails to capture the global properties of the front. %In particular, its inability to deal with topological changes makes it impractical in most instances. For this reason mostly, the LSM quickly superseded this formulation when it was introduced. 
Nevertheless we believe that this approach can be used near regions where $F$ vanishes.

\subsection{Heuristics} \label{subsec:Heuristics} We first present an argument in the smooth setting. Consider the solution $\phi$ to IVP (\ref{eq:LSE}). Suppose $\phi \in C^{1}(\xvec_{0},t_{0})$ and $\phi(\xvec_{0},t_{0})=0$. Assume furthermore that $\phi_{x}(\xvec_{0},t_{0}) \neq 0$, so that the mapping is locally invertible. From the Implicit Function Theorem there exist open neighbourhoods $(\xvec_{0},t_{0}) \in \mathcal{V}$ and $\UU \subset \RR \times [0,T]$, as well as a function 
\begin{eqnarray}
\psi: \UU \longrightarrow \RR ~,
\quad 
\psi: (y,t) \mapsto x=\psi(y,t) ~,
\quad 
\psi \in C^{1}(\UU) ~,
\quad 
(\psi(y,t),y,t) \in \mathcal{V}
\end{eqnarray}
satisfying $\phi(\psi(y,t),y,t) = 0$ $\forall ~ (y,t) \in \UU$. Taking full derivatives of $\phi$ with respect to $y$ and $t$, and using the fact that in $\VV$, $\phi$ satisfies the LSE pointwise gives: 
\begin{eqnarray}
(-\phi_{x}\psi_{t})+F\sqrt{\phi^{2}_{x} + (-\phi_{x}\psi_{y})^{2}} = 0 
\quad \Longleftrightarrow \quad
-\psi_{t} \pm F\sqrt{1 + \psi^{2}_{y}} = 0
\end{eqnarray}
where $\phi_{x}$ and $F$ are evaluated at $(x,y,t)=(\psi(y,t),y,t)$. The sign used in the last equation depends on $\phi_{x} = \pm \sqrt{\phi^{2}_{x}}$. We let $a:= -\mathrm{sign}(\phi_{x}(\xvec_{0},t_{0}))$. 

Now, let $\psi$ satisfy the following Initial Value Problem:  
\begin{eqnarray}
\label{eq:Sideways}
\left\{ \begin{array}{cl} 
\psi_{t} + a F(\psi,y,t) \sqrt{1+\psi^{2}_{y}} = 0 & ~ \mathrm{on}~ \UU \cap \left( \RR \times (t_{0},T) \right)\\
\psi(y,t_{0}) = \psi_{0}(y) & ~ \mathrm{on}~  \UU \cap \left( \RR \times \{ t_{0} \} \right)
\end{array} \right.
\end{eqnarray} 
where $\psi_{0}$ is chosen such that $\phi(\psi_{0}(y),y,t_{0})=0$. Then for all $t \in (t_{0},T)$ the set $\Gamma_{t}$ locally describes the curve at time $t$, i.e., $\Gamma_{t} = \CC_{t}\cap \mathcal{V}$. We now investigate the case where $\MM$ is merely $C^{0}$. For simplicity, we work with $t_{0}=0$. 

\paragraph{Remark} Applying the same argument assuming $\phi_{t}(\xvec_{0},t_{0}) \neq 0$ allows one to formally relate the LSE to the Eikonal equation \cite{OsherSethian}: 
\begin{eqnarray}
\phi_{t}+F(x,y,\psi) \sqrt{(-\phi_{t}\psi_{x})^2 + (-\phi_{t}\psi_{y})^2} = 0
\quad \Longleftrightarrow \quad
||\nabla \psi|| = \frac{-\mathrm{sign}({\phi}_{t})}{F(x,y,\psi)}
\end{eqnarray}
But since by the LSE we have $a:= -\mathrm{sign}(\phi_{t}) = \mathrm{sign}\left(F(x,y,\psi)\right)$, this simplifies to $||\nabla \psi|| = \frac{1}{|F(x,y,\psi)|}$.

%%%%%%%%%%%%%%%%%%%%%%%%%%%%%%%%%%%
%\newpage

\subsection{Theory}

Equation (\ref{eq:Sideways}) is a Cauchy problem of the form 
\begin{eqnarray}
\label{eq:Cauchy}
\left\{ \begin{array}{cl} 
\psi_{t} +H(y,t,\psi,\psi_{y}) = 0 & ~ \mathrm{on}~  \UU \cap \left( \RR \times (0,T) \right) \\
\psi(y,0) = \psi_{0}(y) & ~ \mathrm{on}~ \UU \cap \left( \RR \times \{ 0 \} \right)
\end{array} \right.
\end{eqnarray} 
where the Hamiltonian $H:\RR^{1} \times (0,T) \times \RR \times \RR^{1} \rightarrow \RR$ is defined as $H(y,t,\psi,\psi_{y}) = aF(\psi,y,t)\sqrt{1+\psi^{2}_{y}} $. The function $\psi_{0}$ is defined such that for all $y \in \UU \cap \left( \RR \times \{ 0 \} \right)$ we have $(\psi_{0}(y),y) \in \CC_{0}$. We resort to the rich theory of viscosity solutions of Hamilton-Jacobi equations to study various properties of this problem 
\cite{bardi2008optimal,
BarlesExistence,
MR732102,
UserGuideViscosity,
CrandallLions,
evans2010partial, 
koike2004beginner,
souganidisExistence,
subbotin1994generalized}.
We first address the well-posedness of the PDE. It is a simple matter to verify that the assumptions on $H$ required to apply Theorem 1.1 in \cite{souganidis1985approximation} hold in our context.\footnote{with the exception of (H3) in \cite{souganidis1985approximation}. However, it may be modified to get $\gamma_{R,P} \in \RR$ if $p\in B_{N}(0,P)$ for some $P>0$. } This yields 

\begin{theorem}[Existence \& Uniqueness] 
\label{thm:ExistenceUniqueness}
There exists a unique viscosity solution $\psi$ to problem (\ref{eq:Cauchy}).
\end{theorem} 

We next verify that (\ref{eq:Cauchy}) does have the geometric interpretation advertised in the previous section. 

\begin{theorem}[$\Gamma_{t}$ locally describes $\CC_{t}$]
The set $\Gamma_{t}$ enjoys the following property: $\Gamma_{t} = \CC_{t} \cap \mathcal{V}$. 
\end{theorem} 

\begin{proof} 
Consider IVP (\ref{eq:LSE}) again:
\begin{eqnarray}
\left\{ \begin{array}{rcll} 
\phi_{t}+F|\nabla\phi| &=& 0 & \quad \mathrm{on}~ \RR^{2} \times (0,T) \\
\phi(\xvec,0) &=& \phi_{0}(\xvec) & \quad \mathrm{on}~ \RR^{2} \times \{ 0 \}
\end{array} \right.
\end{eqnarray}
Since it is known that $\xvec \in \CC_{t} \cap \mathcal{V}$ if and only if $\phi(\xvec,t)=0$, we may prove the theorem by showing that: $\xvec \in \Gamma_{t}$ if and only if $\phi(\xvec,t)=0$.  

\fbox{$\Longrightarrow$} We argue by contradiction. Suppose the set $\mathcal{T} = \{ T>t>0 : \exists \xvec \in \Gamma_{t} \mathrm{~s.t.~} \phi(\xvec,t) \neq 0 \}$ is not empty and define $t^{\ast} = \inf \mathcal{T}$. Since $\phi$ is continuous, $\mathcal{T}$ is open and $t^{\ast} \not \in \mathcal{T}$. Therefore, for all $\xvec^{\ast} \in \Gamma_{t^{\ast}}$, $\phi(\xvec^{\ast},t^{\ast})=0$, but for any $\epsilon>0$ sufficiently small, there exists $\xvec_{\epsilon} \in \Gamma_{t+\epsilon}$ such that $\phi(\xvec_{\epsilon},t+\epsilon) \neq 0$. If $\MM$ is differentiable at $(\xvec^{\ast},t^{\ast})$, this contradicts the argument presented in \S \ref{subsec:Heuristics}: The Implicit Function Theorem guarantees that the set $\mathcal{V}$ is open. If $\MM$ is not differentiable at $(\xvec^{\ast},t^{\ast})$, then fix $\epsilon$ and for $\delta >0$ consider $\xvec^{0} \in \Gamma_{t+\epsilon}$ such that $\|\xvec_{\epsilon}-\xvec^{0} \| \leq \delta$ and $\MM$ is differentiable at $(\xvec^{0},t+\epsilon)$. For any $\delta$, such a point can be found since for any $T>t+\epsilon>0$ the singularities of $\Gamma_{t+\epsilon}$ are subsets of measure 0.\footnote{This follows directly from the fact that Problem (\ref{eq:Cauchy}) is a first order Hamilton-Jacobi equation.} Again, the Implicit Function Theorem guarantees that there is a neighbourhood $\tilde{\mathcal{V}}$ of $(\xvec^{0},t+\epsilon)$ where $\phi(\xvec,t)=0$ for any $\xvec \in \Gamma_{t} \cap \tilde{\mathcal{V}}$. Considering the sequence $\delta_{n} = \{ \frac{1}{n} : n \in \NN \}$ and the corresponding sequence $\{ \xvec^{n} \}^{\infty}_{n=1}$, we arrive at the conclusion that $\phi(\xvec_{\epsilon},t+\epsilon) \neq 0$ contradicts the continuity of $\phi$. 

\fbox{$\Longleftarrow$} Assume that there exists $(\xvec,t) \in \mathcal{V}$ such that $\phi(\xvec,t)=0$, but there is no $y$ such that $\xvec=(\psi(y),y) \in \Gamma_{t}$. We re-use the arguments given in the proof of \fbox{$\Longrightarrow$}: If $\MM$ is differentiable at $(\xvec,t)$ then this contradicts the argument in \S \ref{subsec:Heuristics}. If $\MM$ is not differentiable at $(\xvec,t)$, then we can find a sequence $\xvec^{n} \in \Gamma_{t}$ converging to $\xvec$ such that $\phi(\xvec^{n},t)=0$, and obtain the contradiction that $\psi$ is not continuous. \qquad \end{proof}

%%%%%%%%%%%%%%%%%%%%%%%%%%%%%%%%%%%

\subsection{Generalizations}
\label{subsec:Generalizations}

More generally, the above arguments can be applied to yield that there exists neighbourhoods $(\xvec_{0},t_{0}) \in \mathcal{V}$ and $\UU \subset \RR \times [0,T]$, as well as a unique function $\psi: \UU \longrightarrow \RR$, $\psi: (z,t) \mapsto w=\psi(z,t)$ with $(w \cos(\theta)+z,w \sin(\theta)+z,t) \in \mathcal{V}$ satisfying 
\begin{eqnarray}
\label{eq:Skewed}
\left\{ \begin{array}{cl} 
\psi_{t} + a F \left( w \cos(\theta)+z,w \sin(\theta)+z,t \right) \sqrt{\psi^{2}_{z}+1} = 0 & ~ \mathrm{on}~ \UU \cap \left( \RR \times (0,T) \right) \\
\psi(z,0) = \psi_{0}(z) & ~ \mathrm{on}~ \UU \cap \left( \RR \times \{ 0 \} \right)
\end{array} \right.
\end{eqnarray} 
in the viscosity sense. Here $\theta$ is the polar angle of $\xvec_{0}$, $a = -\mathrm{sign} ( \xvec_{0} \cdot \ntwo(\xvec_{0},t_{0}))$ and $\psi_{0}$ is chosen such that for all $z \in \UU \cap \left( \RR \times \{ 0 \} \right)$, we have $(\psi_{0}(z) \cos(\theta)+z,\psi_{0}(z) \sin(\theta)+z) \in \CC_{0}$. When $\theta = 0$ we recover Problem (\ref{eq:Sideways}), whereas when $\theta = \pi/2$, we get that $\psi: (x,t) \mapsto y=\psi(x,t)$ with $(x,\psi(x,t),t) \in \mathcal{V}$ satisfies:
\begin{eqnarray}
\label{xtproblem}
\left\{ \begin{array}{cl} 
\psi_{t} + a F(x,\psi,t) \sqrt{\psi^{2}_{x}+1} = 0 & ~ \mathrm{on}~  \UU \cap \left( \RR \times (0,T) \right) \\
\psi(x,0) = \psi_{0}(x) & ~ \mathrm{on}~ \UU \cap \left( \RR \times \{ 0 \} \right)
\end{array} \right.
\end{eqnarray} 
in the viscosity sense. In subsequent sections, we will refer to Problems (\ref{eq:Sideways}) and (\ref{xtproblem}) as the $yt$- and $xt$-representations of $\MM$, whereas Problem (\ref{eq:Skewed}) will be the \emph{skewed} representation. Those problems provide \emph{sideways} representations of the evolving front. For clarity, remarks pertaining to those will usually be made for the special case of Problem (\ref{eq:Sideways}). 
%\paragraph{Remark} \textbf{REMOVE THIS REMARK.} We may be even less restrictive. In theory the functions $\psi$ may depend on any two variables $(\sigma,\tau)$, where $\tau \neq t$. i.e., The domain of the sideways functions need not be a vertical plane. However, for simplicity, we restrict ourselves to the case $\tau=t$ in this paper. 

\subsection{Discretization}
\label{subsubsec:Discretization} 

Finite-differences schemes for problems such as (\ref{eq:Cauchy}) have been discussed
\cite{crandall1984two, 
CrandallTartar,
souganidis1985approximation}.
Based on these works, we propose the following discretization for Equation (\ref{eq:Sideways}). In this subsection only, we will distinguish between the continuous function $\psi$, and its discrete approximation which we denote as $\chi$. The spatial derivative $\chi_{y}$ must be computed in an upwind fashion. To this end, we introduce the one-sided operators 
\begin{eqnarray}
D^{+}_{l}\chi^{r}  := \frac{\chi^{r}_{l+1} - \chi^{r}_{l}}{h}
\qquad \qquad
D^{-}_{l}\chi^{r}  := \frac{\chi^{r}_{l} - \chi^{r}_{l-1}}{h}
\end{eqnarray}
and suggest:
\begin{eqnarray}
\chi^{r+1}_{l} = \chi^{r}_{l} - a \cdot \Delta t \cdot F(\chi^{r}_{l},y_{l},t^{r}) \cdot \sqrt{1+ \mathrm{upw}(\chi^{r},l,r,\alpha)} 
\end{eqnarray}
where
\begin{eqnarray}
\mathrm{upw}(\chi^{r},l,n,\alpha) &:=& 
\max \{ \alpha,0 \} 
\left( \min \left\{
D^{+}_{l}\chi^{r},0 \right\}^{2} + \max \left\{ D^{-}_{l}\chi^{r},0 \right\}^{2} \right) \nonumber \\
&~& 
-\min \{ 0, \alpha \} 
\left( \max \left\{
D^{+}_{l}\chi^{r},0 \right\}^{2} + \min \left\{ D^{-}_{l}\chi^{r},0 \right\}^{2} \right)
\end{eqnarray}
The constant $\alpha$ acts as a switch and is defined as $\alpha = \mathrm{sign} \left( aF(\chi^{r}_{l},y_{l},t^{r} ) \right)$. 

\begin{proposition}{(Convergence.)}
\label{claim:convergence}
Let $M$ be defined as the local bound on $F$, i.e.,  $M^{r}_{l} = \sup_{(x,y,t) \in B(p^{r}_{l},2h)}\{ |F(x,y,t)| \}$, where $p^{r}_{l} = (\chi^{r}_{l},y_{l},t^{r})$. Assume that $\max \left\{ | D^{+}_{l}\chi^{r} |, | D^{-}_{l}\chi^{r} | \right\} \leq P$ for all $l \in L$ and $0\leq r \leq R$. Suppose $\Delta t$ satisfies
\begin{eqnarray}
\label{CFLbound}
M^{r}_{l} \cdot \Delta t \leq \frac{h}{2P} 
\end{eqnarray}
Then the above scheme is such that $\chi \rightarrow \psi$ as $h$ and $\Delta t \rightarrow 0$, with rate 
\begin{eqnarray}
\| \chi - \psi \|_{\infty} \leq c \sqrt{\Delta t}
\end{eqnarray}
for all $l$, where the constant $c$ depends on $\| \psi_{0} \|$, $\| D\psi_{0} \|$, the numerical Hamiltonian $g$, and $R\Delta t$ where $0 \leq r \leq R$. 

\end{proposition}

\begin{proof}
% Note to myself: I printed this article \cite{souganidis1985approximation} and took useful notes in it. Should refer to it when writing my thesis!  
We proceed by showing that the scheme is \textsl{monotone} and \textsl{consistent} in the sense of \cite{souganidis1985approximation}. The results then follow from Theorem 3.1 of that same paper. The scheme can be rewritten as 
\begin{eqnarray}
\chi^{r+1}_{l} = \chi^{r}_{l} - \Delta t \cdot g \left(y_{l},t^{r},\chi^{r}_{l},D^{+}_{l}\chi^{r},D^{-}_{l}\chi^{r} \right)
\end{eqnarray}
where the numerical Hamiltonian $g$ is easily verified to be consistent, i.e., 
\begin{eqnarray}
g \left(y,t, s, \delta , \delta \right) = H(y,t,s,\delta) \qquad \forall (y,t) \in \UU,~ s \in \RR,~|\delta|<P
\end{eqnarray}
We verify monotonicity by showing that 
%\begin{eqnarray}
%\tilde{G}(t^{n},\Delta t,u,\chi^{n})(y_{l})
%=  \chi^{n}_{l} - a \cdot \Delta t \cdot F(u,y_{l},t^{n}) \cdot \sqrt{1+ \mathrm{upw}(\chi^{n},l,n,\alpha)} 
%\end{eqnarray}
%is a non-decreasing function of $\chi^{n}$. It is sufficient to check that 
the function 
\begin{eqnarray}
G(\chi^{r}_{l-1},\chi^{r}_{l},\chi^{r}_{l+1})
=  \chi^{r}_{l} - a \cdot \Delta t \cdot F(u,y_{l},t^{r}) \cdot \sqrt{1+ \mathrm{upw}(\chi^{r},l,r,\alpha)} 
\end{eqnarray}
is a non-decreasing function of each of its argument, for fixed $u$, $y_{l}$ and $t^{r}$. We only treat the case $\alpha >0$, since the other case is symmetric. Writing $F=F(u,y_{l},t^{r})$ for short gives
\begin{eqnarray}
G(b,c,d) = \left\{ \begin{array}{ll}
c - a \Delta t ~ F \sqrt{1+\left( \frac{d-c}{h} \right)^{2}} & \mathrm{if}~ d-c<0, ~ c-b<0 \\
c - a \Delta t ~ F \sqrt{1+\left( \frac{c-b}{h} \right)^{2}} & \mathrm{if}~ d-c>0, ~ c-b>0 \\ 
c - a \Delta t ~ F \sqrt{1+\left( \frac{d-c}{h} \right)^{2}+\left( \frac{c-b}{h} \right)^{2}} & \mathrm{if}~ d-c<0, ~ c-b>0  \\
c - a \Delta t ~ F & \mathrm{if}~ d-c>0, ~ c-b<0 
\end{array} \right. 
\end{eqnarray}
For the first case: $G_{b}$, $G_{d} \geq 0$ are trivial to check while $G_{c} \geq 0$ only if 
\begin{eqnarray}
1 \geq \left( F^{2} \left( \frac{\Delta t}{h} \right)^{2}-1 \right) \left( - \frac{d-c}{h}\right)^{2}
\quad \Longleftarrow \quad
\frac{\sqrt{1+P^{2}}}{P} \geq M^{r}_{l} \frac{\Delta t}{h}
\end{eqnarray}
Case 2 yields the same condition, whereas Case 3 gives the more restrictive one present in the assumption of the claim. Case 4 is trivial. \qquad \end{proof}

\begin{proposition}{(Stability.)}
\label{claim:stability}
The above scheme is stable, provided that 
\begin{eqnarray}
\label{UpperBounds}
\Delta t < \min \left\{ \frac{h}{2PM^{r}_{l}} ~,~ \frac{P-2}{KP\sqrt{1+2P^{2}}} ~,~ \frac{2}{P\delta} \right\}
\end{eqnarray}
for some $\delta >0$. The constant $P$ is such that $\max \left\{ | D^{+}_{l}\chi^{r} |, | D^{-}_{l}\chi^{r} | \right\} \leq P$ for all $l \in L$ and $0\leq r \leq R$.
\end{proposition}

\begin{proof} Applying Theorem 7 of \cite{Oberman} to our scheme, it is possible to show that for $h$ small enough, the explicit Euler map defined as 
\begin{eqnarray}
S^{l}_{\Delta t}(\chi) = \chi_{l} - a\Delta t \cdot F(\chi_{l},y_{l},t) \sqrt{1+\mathrm{upw}(\chi^{r},l,r,\alpha)}
\end{eqnarray}
is a strict contraction in $\ell_{\infty}$. Bounding $S^{l}_{\Delta t}(\chi) - S^{l}_{\Delta t}(\tau) $ from below (resp.\mbox{~}above) yields the $\second$ (resp.\mbox{~}$\third$) bound in (\ref{UpperBounds}).  \end{proof}

When defining `upw', we implicitly assumed that both $\chi^{r}_{l+1}$ and $\chi^{r}_{l-1}$ were known. In the instance where one of those values is not known, we set $\chi^{r+1}_{l}$ to $+\infty$. Indeed, no value can be assigned to $\chi^{r+1}_{l}$ since it is not possible to infer where the characteristic going through the point $p^{r}_{l}=(\chi^{r}_{l},y_{l},t^{r})$ comes from. %The results just given are assembled to give Algorithm \ref{AlgoSidewaysPDE} presented in \S \ref{sec:AlgoDiscussion}. The effect of setting $\chi^{r+1}_{l}$ to $+\infty$ when $\chi^{r}_{l+1}$ or $\chi^{r}_{l-1}$ is not known is discussed in \S \ref{subsubsec:AlgoGetSideways}. We refer the reader to \cite{Oberman} and \cite{souganidis1985approximation} for guidelines on how to devise other finite-difference schemes. We conclude this section with some remarks. 

\paragraph{Remark}
%In order to assess the well-posedness of Problem (\ref{eq:Cauchy}), as well as prove the convergence and stability of the scheme we propose, we have made the assumption that the norm of the derivatives of $\psi$ and $\chi$ are bounded. For our purposes, this assumption is reasonable since a grid of meshsize $h$ can only resolve adequately functions with $|D\chi|< \mathcal{O}(1/h)$. 
Assuming $P=\mathcal{O}(1/h)$, we may revisit the bounds on $\Delta t$ given in (\ref{UpperBounds}). The first bound  is not very restrictive, even though it scales like $\mathcal{O}(h^{2})$. Indeed $F\approx 0$ implies that $M^{r}_{l}$ should always be small. %\textbf{REMOVE the following sentence.} This will prove to be convenient, see \S \ref{subsubsec:AlgoMainLoopSideways}. 
The bounds imposed by stability are $\mathcal{O}(h)$,
%\begin{eqnarray}
% \frac{P-2}{KP\sqrt{1+2P^{2}}} = \mathcal{O} \left(\frac{1}{P} \right) = \mathcal{O}(h)
%\qquad \mathrm{and} \qquad
% \frac{2}{P\delta} = O(h) 
%\end{eqnarray}
which agrees with the usual CFL number of an advection problem. 

%\paragraph{Remark}
%The authors are aware of the fact that the above results (in particular Theorem \ref{thm:ExistenceUniqueness} \& Proposition \ref{claim:convergence}) most likely apply with far less restrictive conditions on the speed function $F$. How much the assumptions we are making can be relaxed will be the subject of subsequent work. 

%%%%%%%%%%%%%%%%%%%%%%%%%%%%%%%%%%%
%%%%%%%%%%%%%%%%%%%%%%%%%%%%%%%%%%%
%%%%%%%%%%%%%%%%%%%%%%%%%%%%%%%%%%%
%\newpage

\section{Weaving the representations}
\label{sec:Weaving}

Both approaches just discussed in \S \ref{subsec:tFMM} and \S \ref{subsec:FrontFunction} provide methods that locally build the manifold $\MM$. We now address the question of when to use a specific representation. %Given all the computational advantages that the standard FMM (\ie the $xy$-representation) enjoys, it is desirable to use it as the backbone of our method, and only deviate from it when necessary. We take time to discuss some tools used in the algorithms to switch formalism. This is meant to ease the presentation of the full method outlined in \S \ref{sec:AlgoDiscussion}. 

%%%%%%%%%%%%%%%%%%%%%%%%%%%%%%%%%%%
%%%%%%%%%%%%%%%%%%%%%%%%%%%%%%%%%%%

\subsection{The Sign Test} 
\label{subsec:TheSignTest}

Since the approach presented in \S \ref{subsec:tFMM} relies on the assumption that the speed is bounded away from 0, the sign of $F$ is monitored throughout the algorithm. In particular, whenever a point in $(x_{i},y_{j})$ is assigned a value $\psi_{ij}$ using the $(t)$-FMM, the Sign Test is performed as follows. Suppose the point $p_{i-1,j}=(x_{i-1},y_{j},\psi_{i-1,j})$ was used in the computation of $\psi_{ij}$. Considering the line in $xyt$-space joining the point $p_{i-1,j}$ and $p_{ij}$, we check the number of times $d$ that the speed changes sign along this line. If $d=0$, the algorithm can keep running the ($t$-)FMM: The pair $(p_{i-1,j},p_{ij})$ is said to \textsl{pass the Sign Test}. If $d=1$, we should change representation: The pair $(p_{i-1,j},p_{ij})$ \textsl{fails the Sign Test}. If $d>1$, the grid has to be refined. %Note that our assumptions on $|DF|$ outlined in \S \ref{subsec:Assumptions} ensure that we can pick a grid such that we will always have $d\leq 1$. Moreover, in the situation $F=F(x,y)\geq \delta >0$ all pairs always pass the Sign Test so that the standard FMM is used everywhere. 

%%%%%%%%%%%%%%%%%%%%%%%%%%%%%%%%%%%
%%%%%%%%%%%%%%%%%%%%%%%%%%%%%%%%%%%

\subsection[Interpolation]{Conversion of data: Interpolation} 
\label{subsec:Interpolation}

Suppose that the pair $(p_{i-1,j},p_{ij})$ just failed the Sign Test discussed in \S \ref{subsec:TheSignTest}. Then the algorithm must change representation. %This implies that data have to be locally converted so as to be able to use the methodology outlined in \S \ref{subsec:FrontFunction}. 
Without loss of generality, let us suppose that the algorithm switches from the $xy$- to the $yt$-representation. This means the manifold is locally sampled by points of the form $p_{lm} = (x_{l},y_{m},t)$, where $l\in L \subset I$ and $m \in M \subset J$. The $yt$-representation requires points of the form $p^{r}_{m} = (x,y_{m},t^{r})$, where $r \in R \subset K$. See Figure \ref{fig:ConversionOfData}.% for illustrations. %There are many ways to perform this conversion, but one must keep in mind that the manifold may only be $C^{0}$. Any interpolation scheme that is used must therefore capture the singularities of the manifold with at least $\first$ order accuracy. 

\begin{figure}
		\centering
			\includegraphics[width=\textwidth]{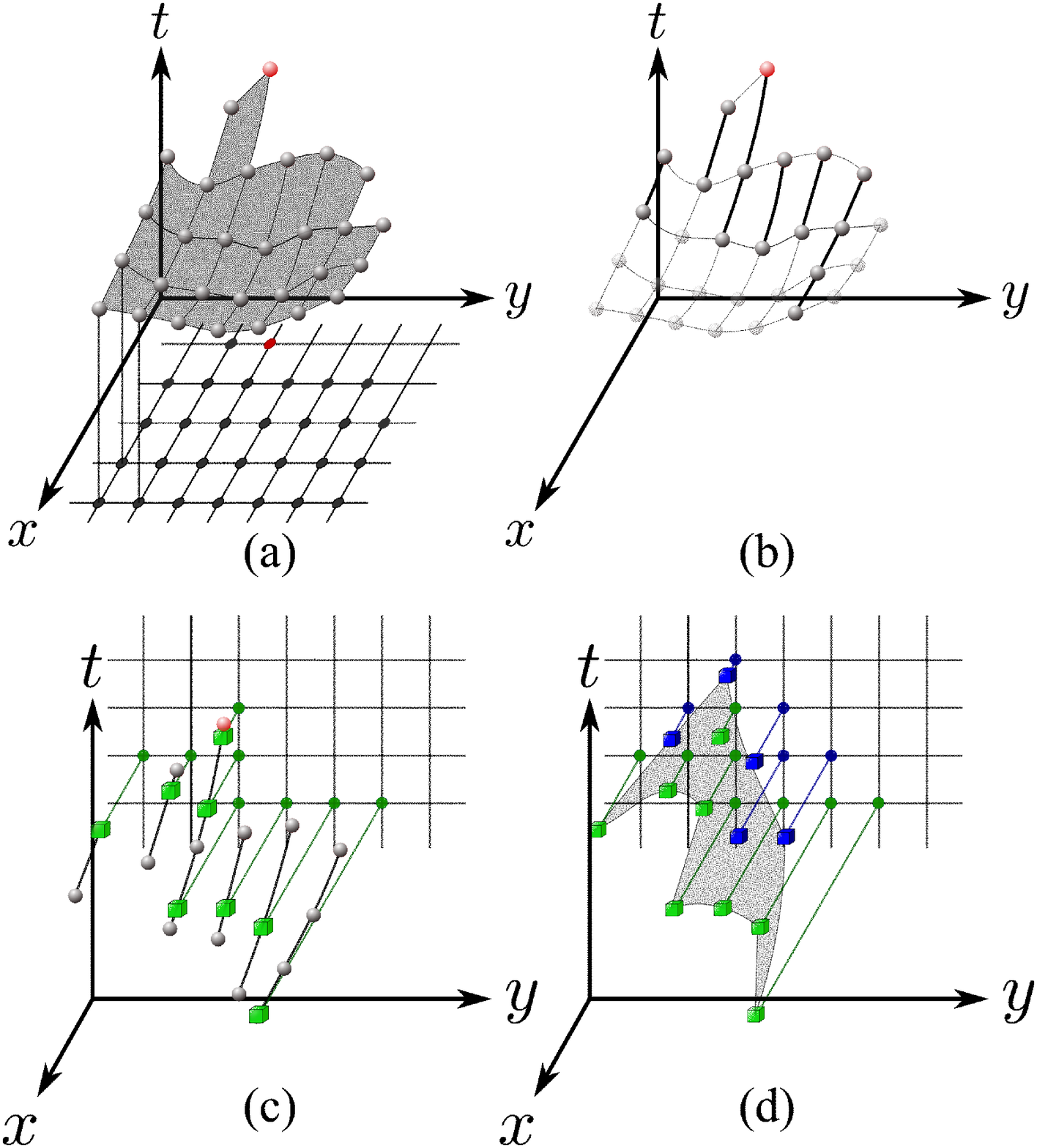}
		\caption{Converting the data using interpolation. (a) Some data in the $xy$-representation. The point $p_{\alpha \beta}$ appears in red. (b) We only keep those points used for interpolation. (c) Performing one-dimensional interpolation line by line, we obtain data in the $yt$-representation. Those are the light green squares. (d) Using the boundary data from (c), the sideways PDE can be solved, to obtain the dark blue squares. By design, the domain shrinks by two points every time step.}
	\label{fig:ConversionOfData}
	\end{figure} 

%%%%%%%%%%%%%%%%%%%%%%%%%%%%%%%%%%%
%%%%%%%%%%%%%%%%%%%%%%%%%%%%%%%%%%%

\subsection{Computing the outward normal}
\label{subsec:OutwardNormal}

Computing the outward normal $\nthree$ accurately at each point sampling $\MM$ is a crucial component of the algorithm. In regions where the level-set function $\phi$ is $C^{1}$, we have $\nthree = \frac{(\phi_{x},\phi_{y},\phi_{t})}{|(\phi_{x},\phi_{y},\phi_{t})|}$. We use the Implicit Function Theorem: If $\psi(y,t)=x$ satisfies $\phi(\psi(y,t),y,t)=0$, then $\phi_{y} = -\phi_{x}\psi_{y}$ and $\phi_{t} = -\phi_{x}\psi_{t}$. Since $\phi_{x} \neq 0$, we set $\vec{n} = (+\mathrm{sign}(\phi_{x}),\psi_{y},\psi_{t})$ and $\nthree = \vec{n}/|\vec{n}|$. We keep track of the normal associated to each point by defining the function 
\begin{eqnarray}
\mathrm{\Norm}:\mathbb{R}^{2} \times \mathbb{R}^{+}\longrightarrow S^{2}
\qquad \qquad
\mathrm{\Norm}(p_{ij})=\nthree(p_{ij})
\end{eqnarray}

%%%%%%%%%%%%%%%%%%%%%%%%%%%%%%%%%%%
%%%%%%%%%%%%%%%%%%%%%%%%%%%%%%%%%%%

\subsection{The Orientation Test}
\label{subsec:TheOrientationTest}

Whenever a point is computed, the algorithm determines the orientation of the outward normal at this point. As explained in \S \ref{subsec:FrontFunction}, this can be done based on the sign of $\nthree_{3}$, the time component of $\nthree$. We define
\begin{eqnarray}
\mathrm{\Orientth}:\mathbb{R}^{2} \times \mathbb{R}^{+}\longrightarrow \{ -1, +1 \}
\qquad \qquad
\mathrm{\Orientth}(p_{ij})= - \mathrm{sign}\left( \nthree_{3} \right)
\end{eqnarray}
The algorithm requires finding which points $p_{ab}$ in a neighbourhood of $p_{ij}$ have the same orientation as $p_{ij}$. This is done using the Orientation Test. A pair $(p_{ij},p_{ab})$ is said to \textsl{pass the Orientation Test} if $\mathrm{\Orientth}(p_{ij})=\mathrm{\Orientth}(p_{ab})$, and to fail it otherwise.

%%%%%%%%%%%%%%%%%%%%%%%%%%%%%%%%%%%%%%%%%%%%%%%%
%%%%%%%%%%%%%%%%%%%%%%%%%%%%%%%%%%%%%%%%%%%%%%%%
%%%%%%%%%%%%%%%%%%%%%%%%%%%%%%%%%%%%%%%%%%%%%%%%

\section{Algorithms \& Discussion} 
%\label{sec:Algo}
\label{sec:AlgoDiscussion}

We introduce some notation before giving the details of the algorithms. We make use of four lists. \Acc ~and \NB ~are lists of triplets, \eg $p_{ij} = (x_{i},y_{j},\psi_{ij})$. \Pp ~and \FA ~are lists of coordinates, \eg $(x_{i},y_{j})$. We define the space and time projection operators as follows: if $p_{ij} = (x_{i},y_{j},\psi_{ij})$, then
\begin{eqnarray}
\pi_{s}: \mathbb{R}^{2} \times \mathbb{R}^{+} \longrightarrow \mathbb{R}^{2} 
&\qquad& 
\pi_{s}(p_{ij}) = (x_{i},y_{j}) \\
\pi_{t}: \mathbb{R}^{2} \times \mathbb{R}^{+} \longrightarrow \mathbb{R}^{+} 
&\qquad& 
\pi_{t}(p_{ij}) = \psi_{ij}
\end{eqnarray}
The following function will be used: 
\begin{eqnarray}
\mathrm{\Grid}:\mathbb{R}^{2}\longrightarrow \mathbb{R}^{+}
\qquad 
\mathrm{\Grid}:(x_{i},y_{j}) \longrightarrow \psi_{ij}
\end{eqnarray}
The set of coordinates $N((x_{i},y_{j})) = \{ (x_{a},y_{b}) : |(i,j)-(a,b)|=1 \}$ consists of the nearest neighbours of $(x_{i},y_{j})$. We use Table \ref{tab:Neighbourhoods} to define two sets of triplets: $\mathrm{NeighEik}((x_{i},y_{j}))$ and $\mathrm{NeighSide}(p_{\alpha \beta})$. The first one is used to compute the value $\psi_{ij}$ in Algorithms \ref{AlgoFMM} and \ref{AlgotFMM}. %We provide a pseudo-code to define this set in Algorithm \ref{AlgoGetNeighbours}. 
Similarly the second set is used in Algorithm \ref{AlgoSidewaysPDE}, where the relevant component of $\nthree(p_{\alpha \beta})$, the normal at $p_{\alpha\beta}$, is denoted by $\eta_{i}$. We are now ready to present the main algorithms. 
%\textbf{REMOVE that sentence, as well as all the details regarding initialization in the appendix.} For completeness, the reader can find a pseudo-code that initializes the data in Appendix \ref{subsec:AlgoInitialization}, along with comments and references. 

\begin{table}
\footnotesize
\begin{center}
\begin{tabular}{|c|c|c|}
\hline
~ & $\mathcal{S} = \mathrm{NeighEik}((x_{i},y_{j}))$ & $\mathcal{S} =\mathrm{NeighSide}(p_{\alpha\beta})$ \\
\hline \hline
$p_{ab} =(x_{a},y_{b},\psi_{ab})$ & 
\bull $(x_{a},y_{b}) \in N((x_{i},y_{j}))$ & 
\bull$(x_{a},y_{b}) \in \{ (x_{l},y_{m}) : l \in L , m \in M \}$ \\
 belongs to $\mathcal{S}$ 
& \bull $p_{ab} \in$ \Acc 
& \bull $p_{ab} \in$ \Acc \\ 
if it satisfies & 
\bull \Grid$(x_{a},y_{b}) = \psi_{ab}$ & 
\bull $\mathrm{sign}(\eta_{i}) = \mathrm{sign}(\nthree_{i})$ \\ 
~ &
\bull $(p_{\alpha \beta},p_{ab})$ passes the Orient.\mbox{}Test & 
where $\nthree = $\Norm$(p_{ab})$  \\
\hline
\end{tabular}
\caption{Definitions of two sets used in Algorithms \ref{AlgotFMM}, \ref{AlgoSidewaysPDE} and \ref{AlgoFMM}}
\label{tab:Neighbourhoods}
\end{center}
\end{table}

%\begin{algorithm}
%\caption{Get $(u_{-},u_{+},v_{-},v_{+})$ using NeighEik$((x_{i},y_{j}))$}\label{AlgoGetNeighbours} 
%\begin{algorithmic}
%	\State $u_{-} \gets \pi_{t}(p_{i-1j})$ if $p_{i-1j} \in \mathrm{NeighEik}((x_{i},y_{j}))$, $+\infty$ otherwise. 
%	\State $u_{+} \gets \pi_{t}(p_{i+1j})$ if $p_{i+1j} \in \mathrm{NeighEik}((x_{i},y_{j}))$, $+\infty$ otherwise. 
%	\State $v_{-} \gets \pi_{t}(p_{ij-1})$ if $p_{ij-1} \in \mathrm{NeighEik}((x_{i},y_{j}))$, $+\infty$ otherwise. 
%	\State $v_{+} \gets \pi_{t}(p_{ij+1})$ if $p_{ij+1} \in \mathrm{NeighEik}((x_{i},y_{j}))$, $+\infty$ otherwise. 
%
%
%\end{algorithmic}
%\end{algorithm}

% $$$$$$$$$$ %

\begin{algorithm}
\caption{Main Loop}\label{AlgoMainLoop}
\begin{algorithmic}[1]

\While{\NB $\neq \emptyset$}

\Statex	\Procedure{Accept a point}{}
	\State $\psi_{\alpha \beta} \gets \min \{ \pi_{t}(p_{ij}) : p_{ij} \in$ \NB $ \}$ 
	\State \Grid $(x_{\alpha},y_{\beta}) \gets \psi_{\alpha \beta}$, ~~$p_{\alpha \beta} \gets (x_{\alpha},y_{\beta},\psi_{\alpha \beta})$
	\State remove $p_{\alpha \beta}$ from \NB $\qquad$ add  $p_{\alpha \beta}$ to \Acc
	\If{$(x_{\alpha},y_{\beta}) \in$ \FA }
		\State remove $(x_{\alpha},y_{\beta})$ from \FA
	\EndIf
\EndProcedure

\Statex	

\If{$\psi_{\alpha \beta}<T$}

	\Procedure{Update Pile}{}
		\ForAll{$(x_{a},y_{b}) \in N((x_{\alpha},y_{\beta}))$}
				\State $\vec{v} \gets  (x_{a},y_{b}) - (x_{\alpha},y_{\beta})$
				\If{$\mathrm{sign}(\vec{v} \cdot \ntwo(p_{\alpha \beta})) = \mathrm{sign}(F(p_{\alpha,\beta}))$ or 0}
					\If{\Grid$(x_{a},y_{b}) =\psi_{ab} < +\infty$}
						\State $p_{ab} \gets (x_{a},y_{b},\psi_{ab})$
						\If{\Orientth$(p_{ab}) \neq$\Orientth$(p_{\alpha \beta})$ }
							\State add $(x_{a},y_{b})$ to \Pp
						\EndIf
					\ElsIf{$(x_{a},y_{b}) \in$ \FA}
						\State add $(x_{a},y_{b})$ to \Pp 
					\EndIf
				\EndIf
		\EndFor
	\EndProcedure

\EndIf

\Statex	\Procedure{Update the Narrow Band}{}
	\ForAll{$(x_{i},y_{j}) \in$ \Pp}
		\State compute $\psi_{ij}$ and $\nthree_{ij}$ using Algo. \ref{AlgoFMM} if $F=F(\xvec)$ or Algo. \ref{AlgotFMM} if $F=F(\xvec,t)$
		\State $p_{ij} \gets (x_{i},y_{j},\psi_{ij})$, ~~\Norm$(p_{ij}) \gets \nthree(p_{ij})$
		\State remove $(x_{i},y_{j})$ from \Pp. 
		\ForAll{$p \in \mathrm{NeighEik}((x_{i},y_{j}))$}
			\State perform the Sign Test for the pair $(p_{ij},p)$
		\EndFor
		\If{at least one pair fails the Sign Test} 
			\State proceed to Algo. \ref{AlgoSideways}, which returns $(k,l,\psi_{kl})$, FAIL and $\nthree$
			\State $p_{ij} \gets (x_{k},y_{l},\psi_{kl})$, ~~\Norm$(p_{ij}) \gets \nthree$, ~~ $i \gets k$,~~ $j \gets l$
		\EndIf
		\State\Orientth$(p_{ij}) \gets - \mathrm{sign}(\nthree_{3}(p_{ij}))$
		\If{FAIL$==0$}
			\If{$\exists~q_{ij} \in $ \NB ~with $\pi_{s}(q_{ij}) = \pi_{s}(p_{ij})$}
				\State remove $q_{ij}$ from \NB
			\EndIf
			\State add $p_{ij}$ to \NB
		\EndIf
	\EndFor
\EndProcedure

\EndWhile
\end{algorithmic}
\end{algorithm}

% $$$$$$$$$$ %

% $$$$$$$$$$ %

\begin{algorithm}
\caption{Solve $|\nabla \psi(x,y)| = \frac{1}{|F(x,y,\psi)|}$}\label{AlgotFMM}
\begin{algorithmic}

%	\State $u_{-} \gets \pi_{t}(p_{i-1j})$ if $p_{i-1j} \in \mathrm{NeighEik}((x_{i},y_{j}))$, $+\infty$ otherwise. 
%	\State $u_{+} \gets \pi_{t}(p_{i+1j})$ if $p_{i+1j} \in \mathrm{NeighEik}((x_{i},y_{j}))$, $+\infty$ otherwise. 
%	\State $v_{-} \gets \pi_{t}(p_{ij-1})$ if $p_{ij-1} \in \mathrm{NeighEik}((x_{i},y_{j}))$, $+\infty$ otherwise. 
%	\State $v_{+} \gets \pi_{t}(p_{ij+1})$ if $p_{ij+1} \in \mathrm{NeighEik}((x_{i},y_{j}))$, $+\infty$ otherwise. 

	\State $u_{\pm} \gets \pi_{t}(p_{i\pm1j})$ if $p_{i\pm1j} \in \mathrm{NeighEik}((x_{i},y_{j}))$, $+\infty$ otherwise. 
	\State $v_{\pm} \gets \pi_{t}(p_{ij\pm1})$ if $p_{ij\pm1} \in \mathrm{NeighEik}((x_{i},y_{j}))$, $+\infty$ otherwise. 
 
	\State $\Theta \gets [0,0,0,0]$

	\For{Quadrant=1\ldots 4}
    
    		\If{Quadrant=1}
        		% Quadrant 1
        		\State $\psi_{v} \gets v_{+}$,~ 
			$\psi_{u} \gets u_{+}$,~ 
			$\tau_{v} \gets \frac{h}{|F(x_{i},y_{j+1},\psi_{v})|}$,~ 
			$\tau_{u} \gets \frac{h}{|F(x_{i+1},y_{j},\psi_{u})|}$,~
			%$s_{v} = +1$,~
			%$s_{u} = +1$
    		\EndIf
	\State (and similarly for other quadrants)
    
%   		\If{Quadrant=2}
%        		% Quadrant 2
%         		\State $\psi_{v} \gets v_{+}$,~ 
%			$\psi_{u} \gets u_{-}$,~ 
%			$\tau_{v} \gets \frac{h}{|F(x_{i},y_{j+1},\psi_{v})|}$,~ 
%			$\tau_{u} \gets \frac{h}{|F(x_{i-1},y_{j},\psi_{u})|}$,~
%			%$s_{v} = +1$,~
%			%$s_{u} = -1$
%    		\EndIf
%    
%
%   		\If{Quadrant=3}
%        		% Quadrant 2
%         		\State $\psi_{v} \gets v_{-}$,~ 
%			$\psi_{u} \gets u_{-}$,~ 
%			$\tau_{v} \gets \frac{h}{|F(x_{i},y_{j-1},\psi_{v})|}$,~ 
%			$\tau_{u} \gets \frac{h}{|F(x_{i-1},y_{j},\psi_{u})|}$,~
%			%$s_{v} = -1$,~
%			%$s_{u} = -1$
%    		\EndIf
%
%   		\If{Quadrant=4}
%        		% Quadrant 2
%         		\State $\psi_{v} \gets v_{-}$,~ 
%			$\psi_{u} \gets u_{+}$,~ 
%			$\tau_{v} \gets \frac{h}{|F(x_{i},y_{j-1},\psi_{v})|}$,~ 
%			$\tau_{u} \gets \frac{h}{|F(x_{i+1},y_{j},\psi_{u})|}$,~
%			%$s_{v} = -1$,~
%			%$s_{u} = +1$
%    		\EndIf
    
    		\If{$(\psi_{v}=+\infty)$ and $(\psi_{u}=+\infty)$}
         		\State $\theta \gets +\infty$,~~ 
			%$\nthree_{1} \gets +\infty$, ~~ 
			%$\nthree_{2} \gets +\infty$
		\Else
		
			\State $\theta \gets \min_{\xi \in [0,1]} \{ \xi \psi_{v} + (1-\xi)\psi_{u} + \sqrt{ \xi^2+(1- \xi)^2} ~\left( \xi\tau_{v}+(1- \xi)\tau_{u} \right) \}$ 
			\State (see Appendix \ref{app:tFMM} for details)
		\EndIf 
    
		\State $\Theta$(Quadrant)$\gets \theta$,~~
	\EndFor

	\State $\psi_{ij} \gets \min(\Theta)$,~~ 
	$Q \gets \mathrm{argmin}(\Theta)$

\end{algorithmic}
\end{algorithm}

%%%%%%%%%%%%%%%%%%%%%%%%%%%%%%%%%%%%%%%%%%%%%%%%
%%%%%%%%%%%%%%%%%%%%%%%%%%%%%%%%%%%%%%%%%%%%%%%%
%%%%%%%%%%%%%%%%%%%%%%%%%%%%%%%%%%%%%%%%%%%%%%%%

%\section{Discussion of the algorithms}
%\label{sec:AlgoDiscussion}

%%%%%%%%%%%%%%%%%%%%%%%%%%%%%%%%%%%%%%%%%%%%%%%%
%%%%%%%%%%%%%%%%%%%%%%%%%%%%%%%%%%%%%%%%%%%%%%%%

\subsection{Algorithm \ref{AlgoMainLoop}, Main loop} 
\label{subsec:AlgoMainLoop}

All steps of the main loop can be checked to be such that if $F=F(x,y) \geq \delta >0$, $\forall (x,y) \in \RR^{2}$, it reduces to the classical FMM. %As desired, this implies that the ($t$-)FMM is the `default' underlying algorithm. 
The sideways formulations are only used when $F\approx 0$. The first procedure, `Accept a point' is identical to the acceptance procedure in the standard FMM \cite{Sethian}, and we therefore omit to discuss it. For clarity, the point accepted during this step is labelled as $p_{\alpha \beta} = (x_{\alpha},y_{\beta},\psi_{\alpha \beta})$ in the rest of the discussion.

%%%%%%%%%%%%%%%%%%%%%%%%%%%%%%%%%%%%%%%%%%%%%%%%

	\subsubsection{Update Pile}
	\label{subsubsec:AlgoUpdatePile}

This step is only performed if $\psi_{\alpha \beta}$ is below a certain predefined time $T$ to ensure that \NB ~is eventually empty. At this stage the algorithm needs to decide whether a nearest neighbour $(x_{a},y_{b})$ of $p_{\alpha \beta}$ should be put in \Pp. To this end three criteria are used: the position, status and orientation of that neighbour. Simply put, \texttt{line 12} has the following effect: If $F(p_{\alpha \beta})>0$ and the considered neighbour lies inside the curve $\CC_{\psi_{\alpha \beta}}$, then the pair $(x_{a},y_{b})$ is not added to \Pp. %This is justified by the same reasons as the ones invoked in \S \ref{subsec:AlgoInitialization}.  
Next the status of this nearest neighbour is considered. If the pair $(x_{a},y_{b})$ was traversed by the curve in the past, then it is only added to \Pp ~if $p_{ab}:= (x_{a},y_{b},$\Grid$(x_{a},y_{b}))$ and $p_{\alpha \beta}$ have different orientations (\texttt{lines 13-16}). Indeed a point in the plane can only be traversed twice if the speed has changed sign in the meantime. If $(x_{a},y_{b})$ is still in \FA, then it is automatically added to \Pp ~(\texttt{lines 17-18}). 

\paragraph{Remark}
The presence of the `if $\psi_{\alpha \beta}<T$' in \texttt{line 8} is in contrast with the standard FMM, where it is proved that since $F\geq \delta >0$, all characteristics exit the domain in finite time. In this context, the size of the computational domain determines $T$. %But once this assumption on $F$ is relaxed, %we are no longer guaranteed that all characteristics exit the domain. Therefore 
%the user must specify the final time $T$. 

%%%%%%%%%%%%%%%%%%%%%%%%%%%%%%%%%%%%%%%%%%%%%%%%

	\subsubsection{Update the Narrow Band}
	\label{subsubsec:AlgoUpdateNB}

This procedure assigns tentative values to the points in \Pp ~using either the standard FMM (see Appendix \ref{subsec:AlgoFMM}) or Algorithm \ref{AlgotFMM}, depending on the domain of $F$. Since $\psi$ only solves the Eikonal equation in regions where $|F| \geq \delta >0$, the first lines of those algorithms ensure that the points involved in the computation of $\psi_{ij}$ all lie in one such region. The steps outlined in \texttt{lines 24-28} represent the main modification to the standard FMM algorithm. The Sign Test is performed to check if the value returned by Algorithm \ref{AlgoFMM} or \ref{AlgotFMM} is valid. If it is not, then Algorithm \ref{AlgoSideways} is called. Using a sideways representation, it attempts to return a point $(x_{k},y_{l},\psi_{kl}) \in \CC_{\psi_{kl}}$. If it manages to do so, note that as explained in \S \ref{subsubsec:AlgoGetSideways}, the triplet returned may not be $(x_{i},y_{j},\psi_{ij})$, which is why $i$ and $j$ are relabelled in \texttt{line 28}. As in the standard FMM, if there already is a point in \NB ~with the same spatial coordinates $(x_{i},y_{j})$, then it is automatically removed from that list. The triplet $(x_{i},y_{j},\psi_{ij})$ is added to \NB. In the event where Algorithm \ref{AlgoSideways} fails, no new point is added to \NB.

%%%%%%%%%%%%%%%%%%%%%%%%%%%%%%%%%%%%%%%%%%%%%%%%
%%%%%%%%%%%%%%%%%%%%%%%%%%%%%%%%%%%%%%%%%%%%%%%%

\subsection{Algorithm \ref{AlgoSideways}, Sideways representation}
\label{subsec:AlgoSolveSideways}

This algorithm is called by the main loop when the speed $F$ is close to 0.

%%%%%%%%%%%%%%%%%%%%%%%%%%%%%%%%%%%%%%%%%%%%%%%%

	\subsubsection{Determine representation}
	\label{subsubsec:AlgoDetermineRep}

In order to work locally, the first step of this procedure defines a square of side length at most $2sh$ for some $s \in \mathbb{N}$ as the new computational grid. Then the representation is chosen based on the normal at $p_{\alpha \beta}$. 

%Unlike the standard FMM, this algorithm computes points that sample the manifold only locally. This is why its very first step is to define a new computational grid to work with. We let it be a square of side length at most $2sh$ for some $s \in \mathbb{N}$. %Note that \texttt{line 3} leaves $L$ and $M$ unchanged if the distance from $(x_{i},y_{j})$ to the side of the domain is greater than $sh$. 
%%It is not obvious how $s$ should be chosen: We come back to this delicate issue in \S \ref{sec:Discussion}. Next the algorithm determines the orientation of the representation. 

%%%%%%%%%%%%%%%%%%%%%%%%%%%%%%%%%%%%%%%%%%%%%%%%

	\subsubsection{Initialization}
	\label{subsubsec:AlgoInitializationSideways}

This is the step where data are converted, as was mentioned in \S \ref{subsec:Interpolation}. The set NeighSide$(p_{\alpha \beta})$ is found; This ensures that the orientation of the points used next is compatible with the current representation. We take time to explain what we mean in \texttt{line 12} in details. It is ideal to build the sideways grid in such a way that the triplet $p_{\alpha \beta}$ is represented exactly on this grid. i.e., For example, if data are being converted to the $yt$-representation, then there should be $\tilde{l} \in L$ and $\tilde{r} \in R$ such that $(y_{\tilde{l}},t^{\tilde{r}}) = (y_{\beta},\psi_{\alpha\beta})$. The function $\psi_{1}:(y,t) \mapsto x$ then satisfies $\psi_{1}(y_{\tilde{l}},t^{\tilde{r}})=x_{\alpha}$, and $(x_{\alpha},y_{\beta},\psi_{\alpha\beta}) = (\psi_{1}(y_{\tilde{l}},t^{\tilde{r}}),y_{\tilde{l}},t^{\tilde{r}})$. This avoids rediscovering the point $p_{\alpha\beta}$ in the procedure `Get $(x_{k},y_{l},\psi_{kl})$' discussed in \S \ref{subsubsec:AlgoGetSideways}. Assigning values to the sideways grid in \texttt{line 13} is an interpolation problem. See Figure \ref{fig:ConversionOfData} (c).

%%%%%%%%%%%%%%%%%%%%%%%%%%%%%%%%%%%%%%%%%%%%%%%%

	\subsubsection{Main loop}
	\label{subsubsec:AlgoMainLoopSideways}

The sideways PDE can now be solved. As mentioned in \S \ref{subsubsec:Discretization}, if either $\psi^{r-1}_{l-1}$ or $\psi^{r-1}_{l+1}$ are set to $+\infty$, then Algorithm \ref{AlgoSidewaysPDE} sets $\psi^{r}_{l}$ to $+\infty$. As depicted on Figure \ref{fig:ConversionOfData} (d), this has the effect of shrinking the size of the set where the PDE is solved: At most $s$ time steps can be taken before all the boundary information available has been used up. When the speed depends on time, we believe that using adaptive time stepping increases the success rate of Algorithm \ref{AlgoSideways}. 
We pick a small $\Delta t$ as long as the speed has not changed sign. This makes the scheme more accurate, thereby increasing the chances of assigning a value to $(x_{i},y_{j})$. Once $F$ changes sign, a large $\Delta t$ is chosen %. %It can be as large as the conditions presented in \S \ref{subsubsec:Discretization} allow.
%Although this deteriorates the accuracy, it increases 
to increase the likelihood of assigning a value to $(x_{\alpha},y_{\beta})$. %Simply put, it goes far enough in time for the manifold to fold and traverse $(x_{\alpha},y_{\beta})$ again. 

%\begin{figure}
%		\centering
%			%\includegraphics[height = 8cm, width = 11cm]{DomainShrinks.png}
%			\includegraphics[width=\textwidth,natheight=556,natwidth=680]{DomainShrinks.png}
%		\caption{Illustration of how the domain where the PDE can be solved shrinks. The PDE is in the $yt$-representation. The point $p_{\alpha \beta}$ that was accepted by Algorithm \ref{AlgoMainLoop} appears as a blue star. The black data are obtained using interpolation. The red data are obtained by solving the sideways PDE.}
%	\label{fig:Shrinking}
%	\end{figure} 

%%%%%%%%%%%%%%%%%%%%%%%%%%%%%%%%%%%%%%%%%%%%%%%%

	\subsubsection{Get $(x_{k},y_{l},p_{kl})$}
	\label{subsubsec:AlgoGetSideways}

Deciding which value is returned by the algorithm is delicate and may be summarized as follows: By default, the algorithm always tries to assign a value to the pair in the \NB ~(\texttt{lines 22-25}). If this is not possible, then it tries to assign a new value to the pair $(x_{\alpha},y_{\beta}) = \pi_{s}(p_{\alpha  \beta})$ (\texttt{lines 26-29}). If this cannot be done either, then this representation failed. The algorithm must attempt using another representation which is chosen based on the ones already attempted. \textit{When the $\first$ attempt fails.} Suppose the $xt$-representation failed, then the algorithm attempts to use the $yt$-representation. \textit{When the $\second$ attempt fails.} Then the scheme resorts to the skewed representation. \textit{When the $\third$ attempt fails.} If the skewed representation also fails, then Algorithm \ref{AlgoSideways} fails entirely. Note that this is expected to happen if $(x_{i},y_{j})$ and $(x_{\alpha},y_{\beta}) \not \in \CC_{t}$ for any $t\in (p_{\alpha \beta},T)$. See Example 2 in \S \ref{sec:Accuracy}.
\paragraph{Remark}
In practice, after each iteration of the for loop \texttt{line 17}, we check if either $(x_{i},y_{j})$ or $(x_{\alpha},y_{\beta})$ has been traversed by the curve. If not, then the for loop keeps going. %This avoids making unnecessary calculations using Algorithm \ref{AlgoSidewaysPDE}. 

% $$$$$$$$$$ %

\begin{algorithm}
\caption{Sideways representation}\label{AlgoSideways}
\begin{algorithmic}[1]

\Procedure{Determine representation}{} 
	\State $s \in \mathbb{N}$ is picked, ~~$\vec{v} \gets (-s,-s+1,\ldots,s-1,s)$, ~~ $L \gets i+\vec{v}$, ~~ $M \gets j+\vec{v}$
	\State $L \gets L \cap I$, ~~ $M \gets M \cap J$
	\If{$|\nthree_{1}(p_{\alpha \beta})|>|\nthree_{2}(p_{\alpha \beta})|$}
		\State use $yt$-representation: $z \gets x$, ~~ $a \gets -\mathrm{sign}(\nthree_{1}(p_{\alpha \beta}))$
	\Else
		\State use $xt$-representation: $z \gets y$, ~~ $a \gets -\mathrm{sign}(\nthree_{2}(p_{\alpha \beta}))$
	\EndIf
\EndProcedure

\State Attempt $ \gets 1$

\While{Attempt$>0$}

\Procedure{Initialization}{} 
	\State get $\mathrm{NeighSide}(p_{\alpha \beta})$ 
	\State the sideways grid $(z_{l},t^{r})$, $l \in L$, $r\in R$ is built 
	\State $\Grid(z_{l},t^{r}) \gets \psi^{r}_{l}$ using interpolation and $\mathrm{NeighSide}(p_{\alpha \beta})$ where possible. 
	\State $\Grid(z_{l},t^{r}) \gets +\infty$ where interpolation cannot be used. 
\EndProcedure

\Procedure{Main loop}{} 
	\If{$a \neq 0$}
		\For{$n=1:R_{\max}$}		
			\State $\Delta t$ is determined
			\For{$l=2:L_{\max}-1$}
				\State compute $\psi^{r}_{l}$ using Algo. \ref{AlgoSidewaysPDE}.
			\EndFor				
		\EndFor
	\EndIf
\EndProcedure

\Procedure{Get $(x_{k}$, $y_{l},\psi_{kl})$}{}
	\If{$(x_{i},y_{j})$ is traversed by the curve}
		\State $\psi_{ij}$ is computed using interpolation
		\State $\psi_{kl} \gets \psi_{ij}$, $x_{k} \gets x_{i}$, $y_{l} \gets y_{j}$, $\nthree(\psi_{kl})$ is computed
		\State Attempt $ \gets 0$, FAIL $\gets 0$
	\ElsIf{$(x_{\alpha},y_{\beta})$ is traversed by the curve}
		\State $\psi_{\alpha \beta}$ is computed using interpolation
		\State $\psi_{kl} \gets \psi_{\alpha \beta}$, $x_{k} \gets x_{\alpha}$, $y_{l} \gets y_{\beta}$, $\nthree(\psi_{kl})$ is computed
		\State Attempt $ \gets 0$, FAIL $\gets 0$
	\Else{~This sideways representation failed.}
		\If{Attempt =1}
			\If{in $xt$-representation}
				\State use $yt$-representation: $z \gets x$, ~~$a \gets -\mathrm{sign}(\nthree_{1}(p_{\alpha \beta}))$
			\EndIf
			\If{in $yt$-representation}
				\State use $xt$-representation: $z \gets y$, ~~$a \gets -\mathrm{sign}(\nthree_{2}(p_{\alpha \beta}))$
			\EndIf 
			\State Attempt = Attempt +1
		\ElsIf{Attempt=2}
			\State use the skewed representation: $z \gets w$, ~~$a \gets -\mathrm{sign}((x_{\alpha}$, $y_{\beta})\cdot \ntwo(p_{\alpha \beta}))$
			\State Attempt = Attempt +1
		\ElsIf{Attempt=3}
			\State Point is not reached before $T$. ~~$\psi_{kl} \gets +\infty$, $x_{k} \gets +\infty$, $y_{l} \gets +\infty$
			\State Attempt $ \gets 0$, FAIL $\gets 1$
		\EndIf			
	\EndIf
\EndProcedure

\EndWhile

\end{algorithmic}
\end{algorithm} 

\clearpage
% $$$$$$$$$$ %

\begin{algorithm}
\caption{Solve $\psi_{t} + a F(\psi,y,t)\sqrt{1+\psi^{2}_{y}} = 0$}\label{AlgoSidewaysPDE}
\begin{algorithmic}[0]
	\If{$(\psi^{r-1}_{l-1}<+\infty)$ \& $(\psi^{r-1}_{l+1}<+\infty)$}
		\State $\alpha \gets$ sign$(aF(\psi^{r-1}_{l},y_{l},t^{r-1}))$
		\State $\psi^{r}_{l} \gets \psi^{r-1}_{l} - a \cdot \Delta t \cdot F(\psi^{r-1}_{l},y_{l},t^{r-1}) \cdot \sqrt{1+ \mathrm{upw}(\psi^{r-1},l,r,\alpha)} $ 
	\Else
		\State $\psi^{r}_{l} \gets +\infty$
	\EndIf
\end{algorithmic}
\end{algorithm}
%%%%%%%%%%%%%%%%%%%%%%%%%%%%%%%%%%%%%%%%%%%%%%%%
%%%%%%%%%%%%%%%%%%%%%%%%%%%%%%%%%%%%%%%%%%%%%%%%
%\newpage

\subsection{General remarks} 
\label{subsec:AlgoGeneralRemarks}

	\subsubsection{Data structure}
	\label{subsubsec:DataStructure}

One of the main differences with the standard FMM is the way we keep track of the various properties associated to each point. The fact that a point $(x_{\alpha},y_{\beta})$ on the plane may be traversed by the curve more than once requires a slightly richer data structure. For example, the functions \Norm and \Orientth have to be defined over triplets rather than over $\RR^{2}$. On the other hand, the lists \Pp~ and \FA~ still consist of coordinates. 
%The set \FA~ consists of those points that have not been traversed by the curve yet. 
Note that when the code ends, \NB~ is empty whereas \FA~ may still contain points. The \Acc~ list may contain multiple triplets sharing the same spatial coordinates. In order to keep track of what the most `up-to-date' value associated with $(x_{a},y_{b})$ is, we make use of \Grid. Indeed, this function enjoys the following property: If there are distinct points $p_{ij},~q_{ij} \in$ \Acc~ such that $\pi_{s}(p_{ij})=\pi_{s}(q_{ij})$, then \Grid$(x_{i},y_{j}) = \max \{  \pi_{t}(p_{ij}),\pi_{t}(q_{ij}) \}$. Viewed as a set, \Grid$(\pi_{s}($\Acc$))$ is the upper semi-continuous envelope of $\MM$. %\textbf{GET RID of that table.} For clarity, we compare the standard FMM and our method in Table \ref{tab:FMMvsMyFMM}. 

%\begin{table}
%\footnotesize
%\begin{center}
%\begin{tabular}{|c|c|c|}
%\hline
%~ & ($t$-)FMM & Our method \\
%\hline \hline
%$F$ & 
%$F \geq \delta >0$ & 
%can vanish and change sign \\
%\Orientth &
%is always +1 & 
%can be +1 or -1 \\
%the Sign Test &
%is always passed & 
%can be failed by some pairs \\
%$T$ is determined &
%by the size of the domain & 
%by the user \\ 
%\hline
%When the code ends: &
%~ & ~ \\
%\hline
%\Pp &
%is empty &
%is empty \\ 
%\FA &
%is empty &
%might not be empty \\ 
%if $(x_{i},y_{j},\psi_{ij}) \in $\Acc, then &
%\Grid$(x_{i},y_{j})=\psi_{ij}$  &
%\Grid$(x_{i},y_{j})$ might not equal $\psi_{ij}$ \\ 
%\hline
%\end{tabular}
%\caption{Comparing the standard FMM and the method we propose. Without loss of generality, we have assumed $F\geq \delta >0$ in the standard FMM case.}
%\label{tab:FMMvsMyFMM}
%\end{center}
%\end{table}

%%%%%%%%%%%%%%%%%%%%%%%%%%%%%%%%%%%%%%%%%%%%%%%%

	\subsubsection{Recovering the curve from $\MM$}
	\label{subsubsec:RecoveringTheCurve}

The set \Acc~ provides a discrete sampling of $\MM$. Using this point cloud, and possibly the normal $\nthree$ to $\MM$ at each point, a continuous representation of $\MM$ can be obtained. See for example \cite{Amenta,SurfaceFromPointCloud,SurfaceFromPointCloud2,Hoppe,MarchingCubes}, and \cite{Zhao}. Given a time $t \in (0,T)$, a contouring algorithm can then be used to find $\CC_{t}$ (see \cite{MarchingCubes}). 

%Together, those algorithms construct a sampling of the manifold $\MM$ but do not provide a mean of recovering the curve $\CC_{t}$, for a given time $t \in (0,T)$. We choose to de-emphasize this final step, since the literature on this subject is rather vast. See for example \cite{Amenta,SurfaceFromPointCloud,SurfaceFromPointCloud2,Hoppe,MarchingCubes}, and \cite{Zhao} for discussions on how to reconstruct codimension one sets from point clouds. We point out however that an approximation of the normal to the surface is known at each triplet sampling $\MM$. This information can be used to significantly improve the accuracy of the reconstruction process. Once a continuous representation of $\MM$ is obtained, a contouring algorithm can be used to find $\CC_{t}$ (see \cite{MarchingCubes}). 

%%%%%%%%%%%%%%%%%%%%%%%%%%%%%%%%%%%%%%%%%%%%%%%%

	\subsubsection{Resolution}
	\label{subsubsec:Resolution}

%The set \Acc~ provides a discrete sampling of $\MM$. 
By construction, the density of points sampling $\MM$ is expected to be lower in regions where $F \approx 0$. 
A remedy to this situation is to also record the points computed in the sideways representations% -- as well as the outward normal at each such point if desired
. 

%A remedy to this situation is to also put the points computed in the sideways representation in \Acc . Recall that such points are of the form $p_{l}^{n} = (w \cos(\theta)+z,w \sin(\theta)+z,t)$ where $\psi: (z_{l},t^{n}) \mapsto w=\psi(z_{l},t^{n})$ and $\theta$ is the polar angle of the point that was accepted in the previous iteration of the main loop. Depending on the method used to get $\CC_{t}$ from $\MM$, the user may also choose to record the outward normal at each of those sideways points as well.

%%%%%%%%%%%%%%%%%%%%%%%%%%%%%%%%%%%%%%%%%%%%%%%%
%%%%%%%%%%%%%%%%%%%%%%%%%%%%%%%%%%%%%%%%%%%%%%%%
%%%%%%%%%%%%%%%%%%%%%%%%%%%%%%%%%%%%%%%%%%%%%%%%

\section{Complexity of the method} 
\label{sec:Complexity}

We derive some estimates for the computational time of the method when $n=2$, \ie two spatial dimensions. Consider a spatial grid of $N^2$ points with meshsize $h$. Let $\Delta t \sim h$, and define $N^{\ast}$ to be the number of gridpoints traversed by $\CC_{t}$ when $0<t<T$. (\ie if a given gridpoint $(x_{i},y_{j})$ is traversed twice, say at times $t_{1}$ and $t_{2}$ where $0<t_{1}<t_{2}<T$, then this contributes $+2$ to $N^{\ast}$.) By construction, the computational time depends on the size of the set $\mathcal{F}_{\MM} := \mathcal{F} \cap \MM$. Indeed, Algorithm \ref{AlgoSideways} is only called when Algorithm \ref{AlgoMainLoop} fails, which occurs whenever an accepted point computed by Algorithm \ref{AlgoMainLoop} is within a spatial distance $h$ of $\mathcal{F}_{\MM}$. Let the number of points computed by Algorithm \ref{AlgoSideways} be $\tilde{N}$. Since the complexity of Algorithm \ref{AlgoMainLoop} is well-known \cite{Sethian}, let us focus on estimating the complexity of a single call to Algorithm \ref{AlgoSideways}. On the square of side $2s$, the Narrow Band forms a one-dimensional subset. Using interpolation to convert the points in a neighbourhood of this set takes $\mathcal{O}(s)$ operations. Algorithm \ref{AlgoSidewaysPDE} makes at most $s^{2}$ operations. Those two steps are performed at most three times. We formally argue that the parameters of the algorithm can be chosen such that this \emph{worse case} complexity is not achieved. The procedure mentioned in the remark of \S \ref{subsubsec:AlgoGetSideways} can be used to prevent Algorithm \ref{AlgoSidewaysPDE} from making unnecessary computations. In \S \ref{subsubsec:AlgoMainLoopSideways}, we explain how using adaptive time-stepping increases the success rate of Algorithm \ref{AlgoSideways}. Moreover, as $N$ increases, the time distance between the accepted point computed by Algorithm \ref{AlgoMainLoop} and $\mathcal{F}_{\MM}$ decreases, which in turn makes Algorithm \ref{AlgoSideways} more successful on average. Altogether, this suggests that the number of attempts taken by Algorithm \ref{AlgoSideways} tends to one for almost all points; this is confirmed by the examples presented in the next section. As a result, the complexity of Algorithm \ref{AlgoSideways} tends to $\mathcal{O}(s)$ for large $N$. 
Given the assumption that $F$ is analytic, we expect $N^{\ast}-\tilde{N} = \mathcal{O}(N^{2})$ and $\tilde{N} = \mathcal{O}(N)$. In practice, the number of points in the local grid $s$ can be chosen as $kN$ for $k \ll 1$. The overall complexity can therefore be estimated as:
\begin{eqnarray}
 \OO (N^{2} \log (N^{2})) + \OO(N) \times \OO(kN) = 
 \underbrace{\OO (N^{2} \log (N^{2}))}_{(t)\mathrm{-FMM}} + \underbrace{\OO(kN^{2})}_{\mathrm{augmented~part}} 
\end{eqnarray}
Note that in the instance where $\mathcal{F}=\emptyset$, we recover the usual complexity of the FMM, namely $ \OO (N^{2} \log (N^{2}))$.

%%%%%%%%%%%%%%%%%%%%%%%%%%%%%%%%%%%%%%%%%%%%%%%%
%%%%%%%%%%%%%%%%%%%%%%%%%%%%%%%%%%%%%%%%%%%%%%%%
%%%%%%%%%%%%%%%%%%%%%%%%%%%%%%%%%%%%%%%%%%%%%%%%

\section{Numerical Tests} 
\label{sec:Accuracy}

In this section, we illustrate how the method works with a variety of examples. We first discuss the methodology used to assess the convergence of the algorithms, and briefly summarize which features and results are expected. We then present the examples. More details are provided in Appendix \ref{app:Details}.

\subsection{Error measurement}
\label{subsec:ErrorMeasurement}

%In the next section, we present various convergence plots. The error associated to each point $p_{ij}$ returned by our scheme is obtained via different procedures, that we now discuss. Those depend on whether an exact representation of the curve is available or not. 
To assess the convergence of our algorithm, we compute the error associated to each point $p_{ij}$ returned by our scheme. %This is done via different procedures that we now discuss. The method used depends on whether an exact representation of the curve is available or not. 

\paragraph{Method 1: $E_{ij}$}
Suppose that an exact solution to the Level-Set Equation (\ref{eq:LSE}), $\phi(x,y,t)>0$ is known, with the property that $|\nabla \phi|=1$ for all $t$. Then evaluating $\phi$ at $p_{ij}=(x_{i},y_{j},\psi_{ij})$ returns the distance to the curve $\CC_{\psi_{ij}}$. We define $E_{ij} = |\phi(p_{ij})|$. This method is used for all examples except Example 4 when $F<0$.

\paragraph{Method 2: $G_{ij}$}
If an exact solution is not available, we get a numerical solution accurate enough to be considered exact. To this end, the Level-Set Equation is solved on a very fine grid using $\second$ order stencils in space, and RK2 in time. At each time step, the zero-contour of $\phi$ is found and sampled. The resulting list of points $\mathcal{B}$ provides a discrete approximation of $\MM$. The error associated to $p_{ij}$ is defined as the smallest three-dimensional distance to this exact cloud of points, \ie $G_{ij} = \min_{q \in \mathcal{B}} \{ |p_{ij}-q| \}$. This method is used for Example 4, when $F<0$.

\subsection{Tests performed}
\label{subsec:TestsPerformed}

\paragraph{Accuracy of Algorithm \ref{AlgoSidewaysPDE}} 

In \S \ref{subsubsec:Discretization}, it is mentioned that the sideways method we propose converges with at least $\mathcal{O}(h^{1/2})$ accuracy. To verify this, we pick a domain $\mathcal{U}$, initialize say $x=\psi(y_{m},t^{0})$ with exact data for some initial time $t^{0}$, and run Algorithm \ref{AlgoSidewaysPDE} for different gridsizes. The result is a subset of $\MM$, encoded as a list of points of the form $p^{r}_{m}=(\psi^{r}_{m},y_{m},t^{r})$. An error is associated to each point $p^{r}_{m}$ such that $\psi^{r}_{m}<\infty$ using either Method 1 or 2, i.e., $E^{r}_{m} = | \psi_{\mathrm{exact}}(y_{m},t^{r}) -  \psi^{r}_{m}|$ or $G^{r}_{m} =  \min_{q \in \mathcal{B}} \{ |p^{r}_{m}-q| \}$. A two-dimensional $L_{1}$ norm is then used to report the results in Figure \ref{fig:SidewaysConvergence}, \eg $L_{1} =h^{2} \cdot \sum_{m \in M} \sum_{r \in R} E^{r}_{m}$.

\paragraph{Accuracy of the full scheme}

When testing the accuracy of the full scheme, we distinguish between different regions of the resulting set \Acc. %One of the reasons for doing so is to better understand the influence of the call to Algorithm \ref{AlgoSideways} on the propagation of the error. 
When studying a region computed by the $(t)$-FMM, a two-dimensional $L_{1}$ norm is used: $L_{1} =h^{2} \cdot \sum_{i \in I} \sum_{j \in J} E_{ij}$. Note that our assumptions on $F$ imply that the points computed using the sideways representations form one-dimensional sets of $\RR^{2}\times [0,T]$. Consequently, a one-dimensional $L_{1}$ norm is used to study those points: $L_{1} =h \cdot \sum_{i \in I} \sum_{j \in J} E_{ij}$. The global error (computed using all the points in \Acc) is a two-dimensional $L_{1}$ norm. It may be interpreted as an approximation of the volume enclosed by the exact and the approximated surfaces. 

We report the $L_{\infty}$ error qualitatively, through the black \& white representations of the set \Acc. Those figures are obtained by computing the relative error at each point, i.e., if $L_{\infty} = \max_{i \in I, ~ j \in J} \{ E_{ij} \}$, then $e_{ij} = E_{ij} / L_{\infty}$; and then shading the point accordingly: The darker a point, the larger its relative error $e_{ij}$. 

%%%%%%%%%%%%%%%%%%%%%%%%%%%%%%%%%%%%%%%%%%%%%%%%

\subsection{Expectations}

By assumption, 
%we expect the sideways scheme to be used to compute points lying in a one-dimensional subset of $\RR^{2} \times [0,T]$. This implies that 
as $h \rightarrow 0$, the 1st order $(t)$-FMM scheme is used almost everywhere. This should reflect in the global error: It should follow the same trend as the $(t)$-FMM%, which is $\first$ order accuracy
. Moreover, we expect the call to Algorithm \ref{AlgoSideways} to increase the constant of convergence. A question that we address is the extent to which this degrades the local and global accuracy. %We address this specific concern in Example 2. 
We investigate the behaviour of the scheme in the presence of shocks \& rarefactions in Example 4, as well as in \S \ref{sec:Discussion}. \\

In all examples but the fourth one, the initial curve $\CC_{0}$ is the circle centred at the origin, with radius $r_{0}=1/4$. In all tests, data are initialized with exact values.

%%%%%%%%%%%%%%%%%%%%%%%%%%%%%%%%%%%
%%%%%%%%%%%%%%%%%%%%%%%%%%%%%%%%%%%

\subsection{Example 1: $F=F(t)= 1-e^{10t-1}$}
\label{subsec:Example1}

The main purpose of this example is to illustrate the basic ideas at play in the method. %We therefore go through the data analysis meticulously. 
The speed is such that we expect the circle to first expand up to time $t=0.1$ and then contract until it collapses to the origin.
We first assess the order of convergence of the method for the sideways representation. 
The results reported on Figure \ref{fig:SidewaysConvergence} clearly indicate that it is $\mathcal{O}(h)$. This is higher than the $\mathcal{O}(h^{1/2})$ rate that was predicted in \S \ref{subsubsec:Discretization}. 
When the entire code is run, the set of \Acc ~points is presented on Figure \ref{fig:Ex1Surface} (a)-(b). %Those qualitative results meet our expectations. 
One-dimensional optimization is used for those points traversed by a characteristic that is almost aligned with one of the spatial axis. We note that the sideways points are computed in the $yt$- (resp.\mbox{~}$xt$-)representation when $\ntwo$ aligns better with the $x$- (resp.\mbox{~}$y$-)axis. As expected, the sampling of the surface is sparser near the plane $t = 0.1$. %The top view (middle plot) reveals some asymmetry in the set: This is due to our choice of grid, which is slightly shifted. 
Remark that in this example, none of the sideways points were computed in the skewed representation. %This turns out to be the case for all the gridsizes considered in this example.
The global convergence results are presented in Figure \ref{fig:Ex1Surface}, (d). We distinguish between the bottom part of the surface, the top part, and those points computed using the sideways representation. %The point of doing so is two-fold. 
On the one hand, the results pertaining to the bottom part allow us to conclude that the $t$-FMM is $\mathcal{O}(h)$, as predicted in \S \ref{subsec:tFMM}. On the other hand, we can study the effect of the call to Algorithm \ref{AlgoSideways} on the behaviour of the scheme. Indeed, although the $t$-FMM also converges with $\mathcal{O}(h)$ when used to build the top part of the surface, it does so with a larger constant. %These results are a bit harder to see, since they superimpose with those for the overall scheme. 
We conclude that changing representation does deteriorate the accuracy of the sampling but only to a mild extent. %The results pertaining to the sideways points agree with those featured on Figure \ref{fig:SidewaysConvergence}. 
To gain a better understanding of where the loss of accuracy from the bottom to the top part stems from, the relative $L_{\infty}$ error $e_{ij}$ associated to each point can be viewed on Figure \ref{fig:Ex1Surface} (c). Those points computed using one-dimensional optimization in the $t$-FMM, just after $t=0.1$ bear the largest errors. Two reasons can explain this: Some of those points are clearly traversed by characteristics that are not aligned with the spatial axes. Nevertheless, the scheme resorts to one-dimensional optimization to assign them values, for lack of a better method. Indeed, when those points are put in \Pp, there are not enough neighbours with negative orientation available to use two-dimensional optimization. We also suspect the constant of convergence of the $t$-FMM to depend on $\delta$ where $|F|\geq \delta >0$. In practice, this method is found to perform poorly when using points $p_{ij}$ such that $F(p_{ij}) \approx 0$. 

\paragraph{Remark}
Those outliers do not degrade the accuracy of the method, even locally. This is because by design, Fast Marching Methods assign values to the points in \Pp ~using only those neighbours with a smaller value. As a result, those outliers are not used in any of the calculations of the values of their neighbours. In practice, it is found that they eventually become isolated points of the \NB ~ before getting accepted.

\begin{figure}
		\centering
			\includegraphics[width=\textwidth]{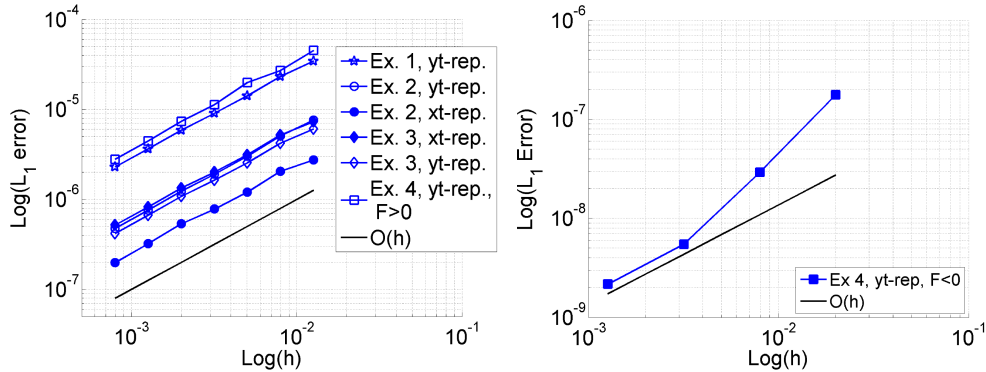}
		\caption{Convergence results for the sideways scheme, using (left) Method 1 (right) Method 2.}
	\label{fig:SidewaysConvergence}
	\end{figure} 

%\begin{figure}
%		\centering
%			%\includegraphics[width=\textwidth,natheight=3568,natwidth=3765]{Ex1BWfriendly.png}
%			\includegraphics[width=\textwidth]{Ex1BWfriendly.eps}
%		\caption{Example 1: (a) - (c) Three different perpectives of the set \Acc. (d)  \Acc ~is featured. The relative error $e_{ij}$ determines the shade of each point.}
%	\label{fig:Ex1Surface}
%\end{figure} 
%
%\begin{figure}
%		\centering
%			%\includegraphics[width=.5\textwidth,natheight=1200,natwidth=1600]{Ex1Convergence.png}
%			\includegraphics[width=.5\textwidth]{Ex1Convergence.eps}
%		\caption{Example 1: Convergence results.}
%	\label{fig:Ex1Convergence}
%\end{figure} 

\begin{figure}
		\centering
			\includegraphics[width=\textwidth]{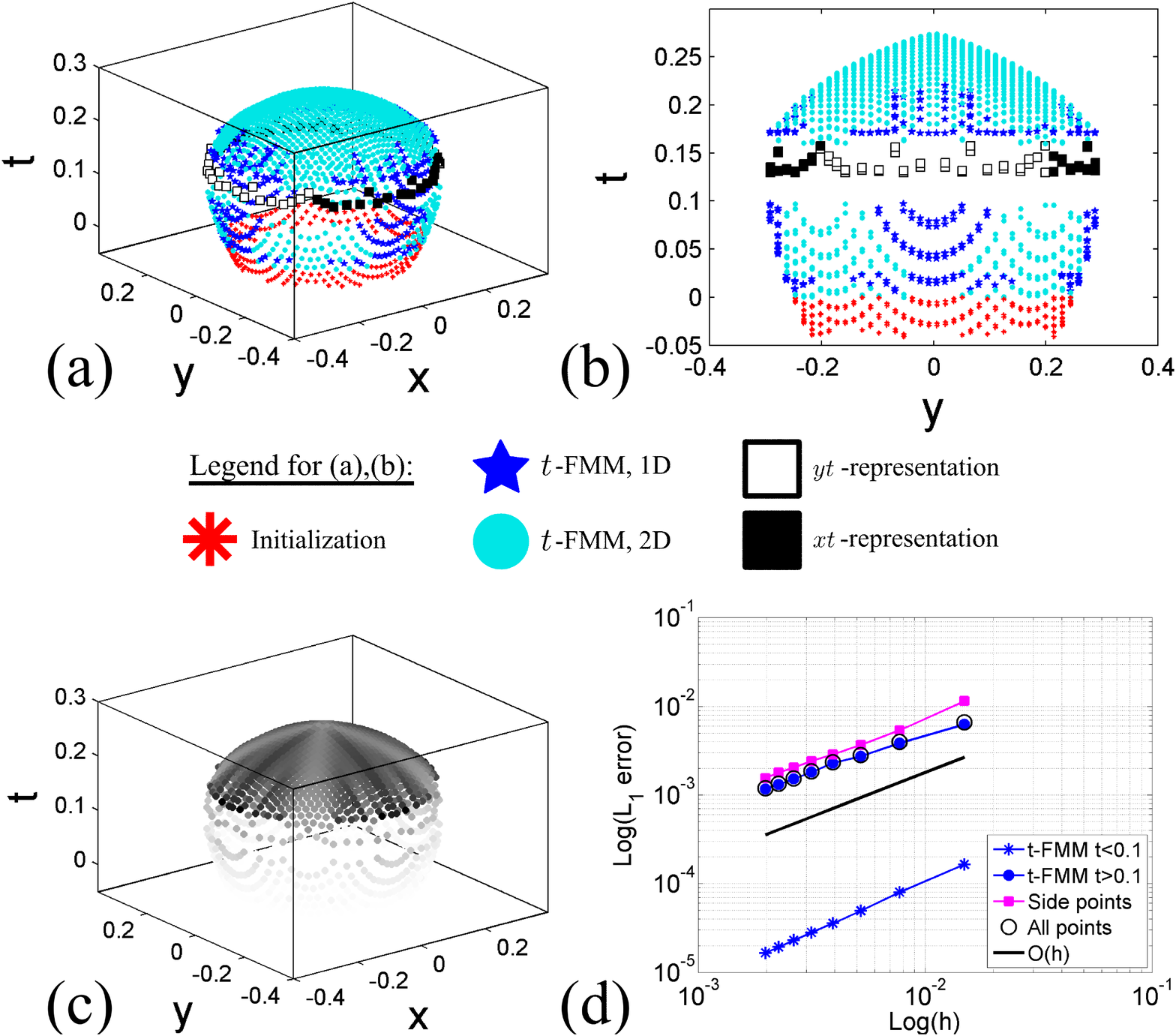}
		\caption{Example 1: (a) - (b) Different perpectives of the set \Acc. (c)  \Acc ~is featured. The relative error $e_{ij}$ determines the shade of each point. (d) Convergence results.}
	\label{fig:Ex1Surface}
\end{figure}

\subsection{Example 2: $F=F(x)=x$}

The given speed is such that the curve remains a circle whose radius grows while its center shifts to the right. Our method adequately handles this case as a single problem, although the speed changes sign across the $y$-axis. As expected, Algorithm \ref{AlgoSideways} fails near the points $(0,0.25)$ and $(0,-0.25)$, as shown on Figure \ref{fig:Ex3Surface} (a). 
%This example illustrates that it is sometimes normal for the code to stop trying to assign values to a point in \Pp. Indeed, referring to Figure \ref{fig:Ex3Surface}, consider what happens to those points with $x<0$, within a distance $h$ of the $y$-axis. As they are accepted, their neighbours go into \Pp. However these neighbours lie to the right of the $y$-axis, and are therefore never reached by the curve. %It is then expected that Algorithm \ref{AlgoSideways} fails to assign a value to those points. The fact that the curve does not change direction of propagation means that no value can be assigned to $(x_{\alpha}, y_{\beta})$ either. So Algorithm \ref{AlgoSideways} fails altogether. This also explains why one-dimensional optimization is used to compute the points at the rim of the surface. 
%The second fact is that this is a typical example where storing the points computed in the sideways representation would significantly increase the resolution of the manifold near $\mathcal{F}$.
The sideways scheme was tested both in the $xt$- and the $yt$-charts, and was found to be $\first$ order in each case (Figure \ref{fig:SidewaysConvergence}). The results for the full scheme show that it converges with $\mathcal{O}(h)$ accuracy everywhere (Figure \ref{fig:Ex3Surface} (b)). %Note that the region $x>0$ bears a larger constant both because it has a lot more points to it, and also because on Figure \ref{fig:Ex3Conv} (a) the error is seen to grow as the points are further away from the initial curve. 
Let us bring up that a bi-directional FMM was proposed in \cite{chopp2009another} to solve a related problem.

\begin{figure}
		\centering
			\includegraphics[width=\textwidth]{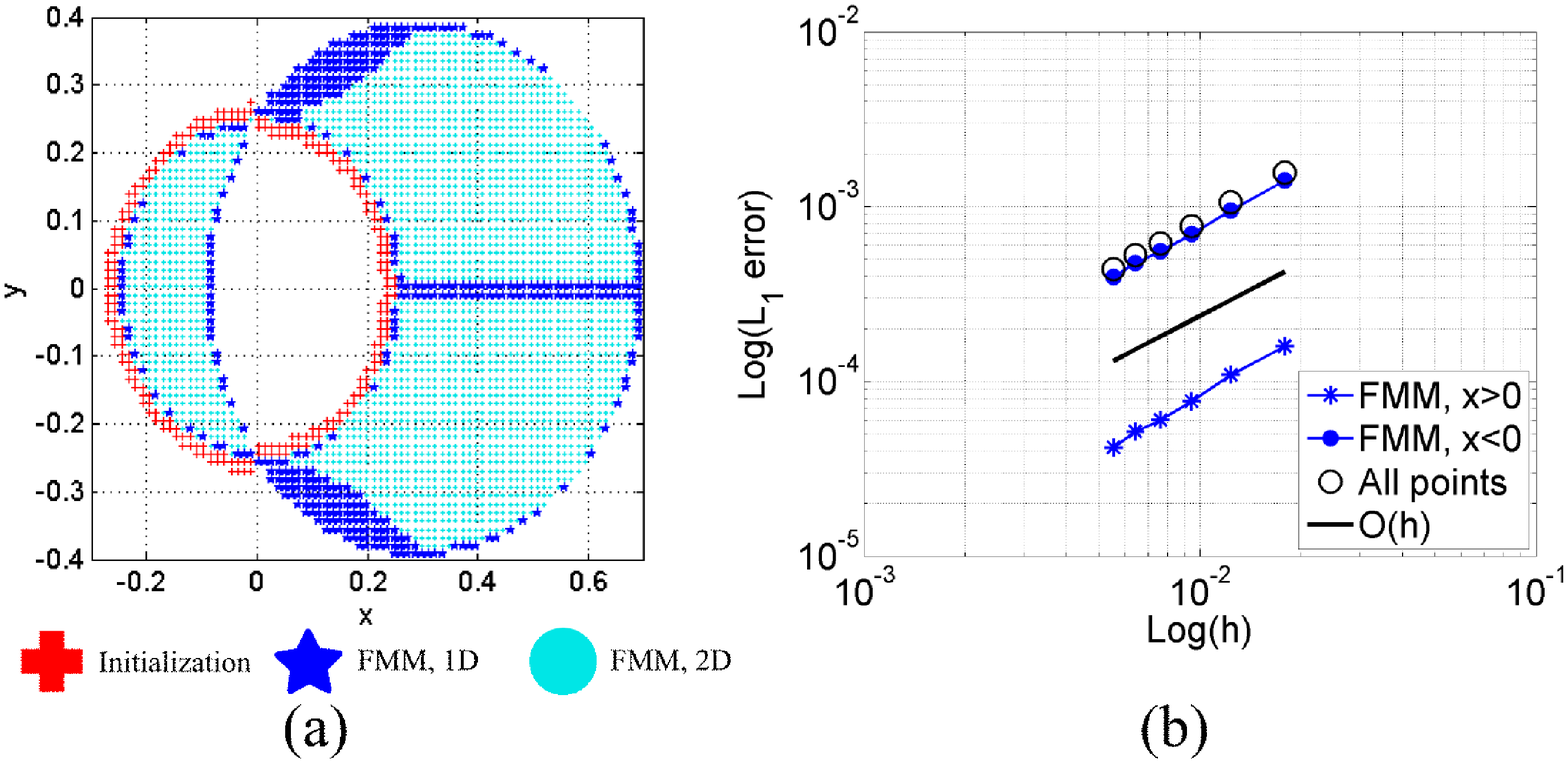}
		\caption{Example 2: (a) The set \Acc ~(b) Convergence results.}
	\label{fig:Ex3Surface}
	\end{figure} 

%\begin{figure}
%		\centering
%			%\includegraphics[width=\textwidth,natheight=775,natwidth=1467]{Ex3Surface.png}
%			\includegraphics[width=\textwidth]{Ex3Surface.eps}
%		\caption{Example 3: The set \Acc.}
%	\label{fig:Ex3Surface}
%	\end{figure} 
%
%\begin{figure}
%		\centering
%			%\includegraphics[width=\textwidth,natheight=777,natwidth=1654]{Ex3LInfAndConvBWfriendly.png}
%			\includegraphics[width=\textwidth]{Ex3LInfAndConvBWfriendly.eps}
%		\caption{Example 3: (a) The set \Acc. The relative error $e_{ij}$ determines the shade of each point. (b) Convergence results.}
%	\label{fig:Ex3Conv}
%	\end{figure} 

%%%%%%%%%%%%%%%%%%%%%%%%%%%%%%%%%%%
%%%%%%%%%%%%%%%%%%%%%%%%%%%%%%%%%%%

\subsection{Example 3: $F=F(x,y,t)$}
(See Appendix \ref{app:Example4} for details about $F$.) 
%This example tackles the more general case where the speed depends on all variables. 
This example differs significantly from the previous ones in that the set $\mathcal{F}$ no longer consists of planes. The exact solution $\CC_{t}$ is a circle that only grows at first, and then starts moving in the positive $x$-direction. Our method is observed to perform very well on this example; We present the resulting surface and the first order convergence results on Figures \ref{fig:SidewaysConvergence} \& \ref{fig:Ex4Conv}. 

\begin{figure}
		\centering
		
			\includegraphics[width=\textwidth]{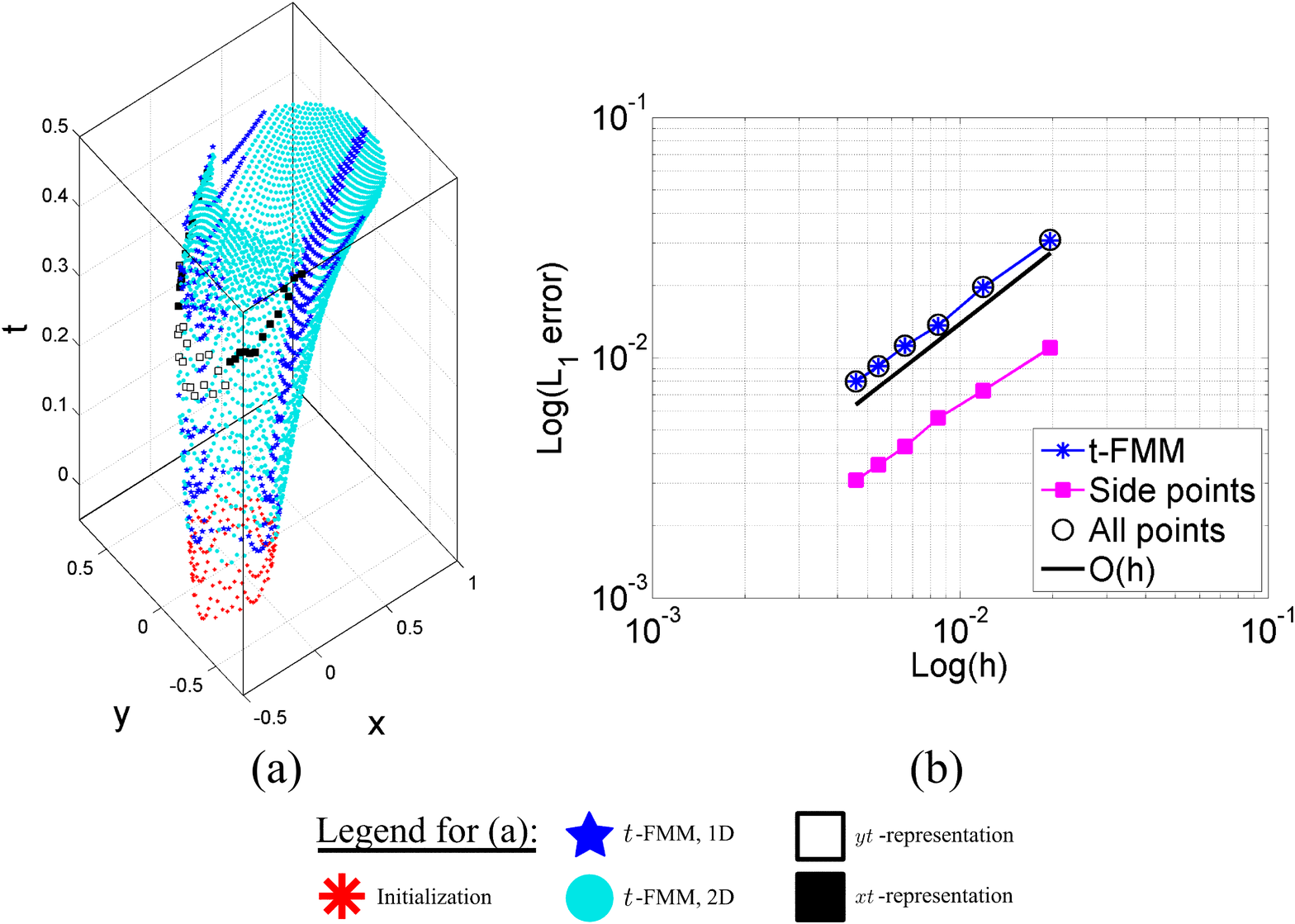}
		\caption{Example 3: (a) The set \Acc ~(b) Convergence results.}
	\label{fig:Ex4Conv}
	\end{figure} 

%\begin{figure}
%		\centering
%			%\includegraphics[width=.5\textwidth,natheight=1200,natwidth=1600]{Ex4Convergence.png}
%			\includegraphics[width=.5\textwidth]{Ex4Convergence.eps}
%		\caption{Example 4: Convergence results.}
%	\label{fig:Ex4Conv}
%	\end{figure} 
%
%
%\begin{figure}
%		\centering
%			%\includegraphics[width=\textwidth,natheight=3425,natwidth=3650]{Ex4SurfaceBWfriendly.png}
%			\includegraphics[width=\textwidth]{Ex4SurfaceBWfriendly.eps}
%		\caption{Example 4: (a) The set \Acc ~(b) The set \Acc. The relative error $e_{ij}$ determines the shade of each point. }
%	\label{fig:Ex4Surface}
%	\end{figure} 

%%%%%%%%%%%%%%%%%%%%%%%%%%%%%%%%%%%
%%%%%%%%%%%%%%%%%%%%%%%%%%%%%%%%%%%

\subsection{Example 4: Two merging circles}

This example tests the ability of the scheme to capture topological changes. 
%Given that the backbone of our method is either the classical or the time-dependent FMM, it should properly handle topological changes. In this example, 
The initial codimension-one manifold consists of two disjoint circles of radius $r_{0}=1/4$, with centres at $(-.3,0)$ and $(.3,0)$. The speed is such that the circles first expand, until they touch and merge. %Note that this gives rise to a singularity on the manifold along the $y$-axis. 
Then the speed changes sign, which makes the curve shrink until it pinches off and splits into two distinct curves. The set \Acc ~is presented in Figure \ref{fig:Ex5Surface} (a).%Plotting \Acc ~reveals that the scheme rightfully uses one-dimensional optimization all along the $y$-axis. %As long as the speed is positive, an exact solution is available. Therefore errors are associated to each point using Method 1 described in \S \ref{subsec:ErrorMeasurement}. Passed this time however, no exact solution is easily computed, so Method 2 is used. This further justifies distinguishing between the error associated to the top and the bottom of $\MM$, as well as presenting the $L_{\infty}$ norm results on two separate sub-figures of \ref{fig:Ex5Surface}. 
The accuracy of the sideways scheme is investigated on a domain that comprises the shock when $F>0$, and the rarefaction when $F<0$. 
First order convergence is obtained in each case (Figure \ref{fig:SidewaysConvergence}). The full scheme also shows $\first$ order convergence (Figure \ref{fig:Ex5Surface} (b)). The convergence of the sideways points and the top part is a little shy of first order, but this can be attributed to the measurement method. Those results demonstrate how robust the overall scheme is. Note that a similar example was tackled in \cite{GFMM}, with a speed $F$ that depended linearly on time. 

\begin{figure}
		\centering
			\includegraphics[width=\textwidth]{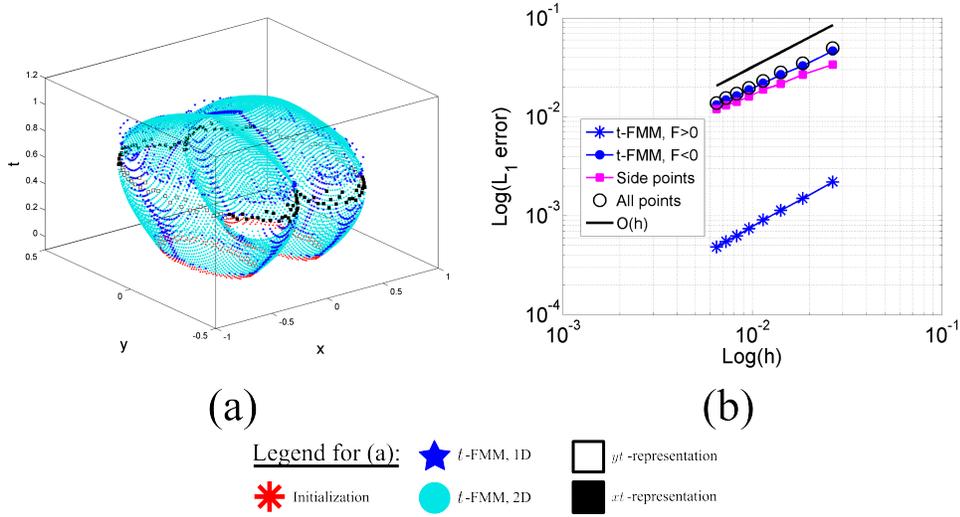}
		\caption{Example 4: (a) The set \Acc. (b) Convergence results.}
	\label{fig:Ex5Surface}
	\end{figure} 

%\begin{figure}
%		\centering
%			%\includegraphics[width=.7\textwidth,natheight=2617,natwidth=3667]{Ex5LInfBWfriendly.png}
%			\includegraphics[width=.7\textwidth]{Ex5LInfBWfriendly.eps}
%		\caption{Example 5: The relative error $e_{ij}$ determines the shade of each point. }
%	\label{fig:Ex5ErrorLInf}
%	\end{figure} 

%%%%%%%%%%%%%%%%%%%%%%%%%%%%%%%%%%%
%%%%%%%%%%%%%%%%%%%%%%%%%%%%%%%%%%%
%%%%%%%%%%%%%%%%%%%%%%%%%%%%%%%%%%%

\section{Discussion}
\label{sec:Discussion}

In the light of the examples presented in the previous section, we address the limitations, weaknesses and advantages of the algorithm.

We illustrate one of the main limitation of the scheme with an ultimate example. The speed is chosen such that the initial circle immediately develops a kink along the $x$-axis at time $t=0$. Its subsequent shape resembles that of an almond slowly turning in the counterclockwise direction while expanding. The sign of the speed changes, forcing the curve to contract while retaining its slanted shape. See Appendix \ref{subsec:AlmondExample} for details. The most prominent feature of this example is that, as is depicted on Figure \ref{fig:Lemon}, the shock is not a straight line. %Moreover, it becomes sharper as time progresses. 
Remark that the speed $F$ does not satisfy the assumptions of this paper outlined in \S \ref{subsec:Assumptions}: It is only a $C^{0}$ function of $\RR^{2} \times [0,T]$. 
The surface that results from running the algorithm at high resolution is shown on Figure \ref{fig:Ex6Surface}. The shock is clearly visible, and has the expected figure-eight shape. Nevertheless some points `escape' through the shock when the speed changes sign, and start out two new fronts that keep on expanding. %Figure \ref{fig:Ex6NarrowBand} shows the set $\mathcal{P}$ at a time close to $T_{F}$. 
The problem stems from the procedure `Update Pile' in Algorithm \ref{AlgoMainLoop}. In order to decide which points go in \Pp, the code distinguishes between the inside and the outside of the curve using the normal $\ntwo$ (cf. \texttt{line 12} of Algorithm \ref{AlgoMainLoop}). Consider what happens along the shock, where $\ntwo$ has a discontinuity. So long as the expansion is outwards, this does not cause problems. But when the direction of propagation reverses, some points that should stay in \FA ~are moved into \Pp . %We illustrate this process in Figure \ref{fig:Ex6NarrowBand}. 
A possible remedy to this issue is to approximate the normal cone along the shock. This additional information could be used as an updating criterion. 

%\paragraph{Remark}
%We do not observe this problem in Example 5, even though there is a similar ambiguity along the shock. This is because if the front duplicated, it would do so inside the two expanding circles. Values already in \Grid ~would be overwritten by this new front, which is impossible unless there has been a change in orientation. Cf. \texttt{line 15} of the procedure Update Pile.  

\begin{figure}
		\centering
			\includegraphics[width=\textwidth]{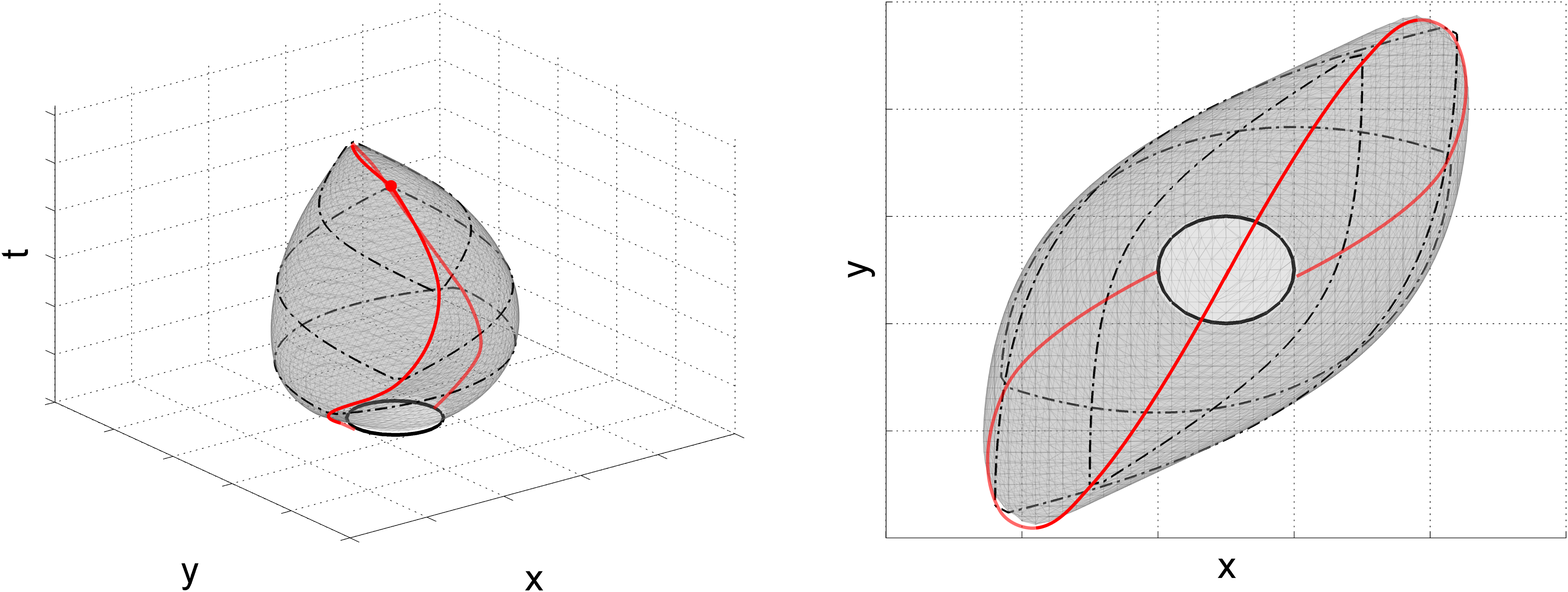}
		\caption{$\MM$ for the almond example. The shock appears as a red plain line.}
	\label{fig:Lemon}
	\end{figure} 

\begin{figure}
		\centering
			\includegraphics[width=\textwidth]{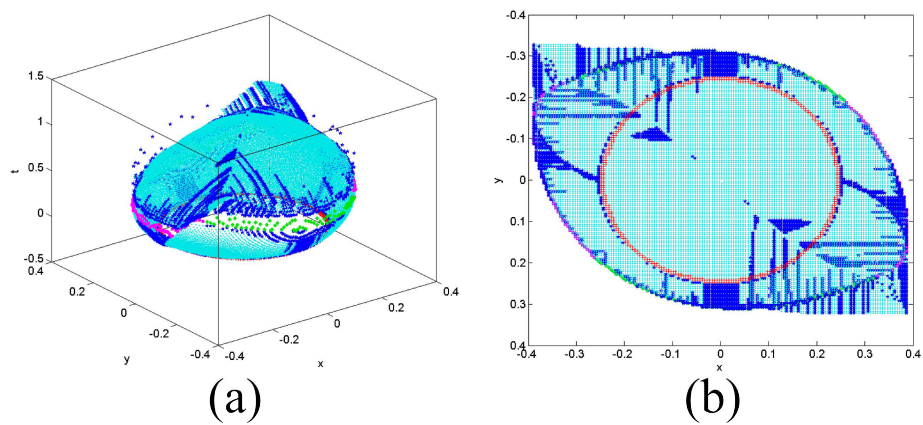}
		\caption{The almond example: The set \Acc ~(a) side view, (b) viewed from above.}
	\label{fig:Ex6Surface}
	\end{figure} 

% This command puts the figures on top of one another, with sublabels (a), (b)...etc. Not what I want though. 
%\begin{figure}[H] \centering
%
%\subfloat[Points]{\centering \includegraphics[height = 6cm, width = 8.76cm]{Ex6NarrowBand3.png}
%
%\label{kCPU1}} \hfil
%
%\subfloat[Circles]{\centering \includegraphics[height = 6cm, width = 6cm]{DuplicatedFrontExplanation.png} \label{kCPU2}}
%
%\caption{(a) The set $\mathcal{P}$ at $t\approx 0.7$. The fronts arising from the shock are clearly visible. (b) Duplication of the front: The dark blue points are in Accepted, the lighter red points are in the Narrow Band. The black point $p_{\alpha\beta}$ has just been acceptedand $F(p_{\alpha \beta})<0$. According to the procedure Update Pile, the circled point underneath it is put in Pile, even though it should not be traversed by the curve when $F<0$.} \label{b1}
%\end{figure}

%\begin{figure}
%		\centering
%			%\includegraphics[width=\textwidth,natheight = 1485, natwidth = 3138]{Ex6DuplicatedFront.png}
%			\includegraphics[width=\textwidth]{Ex6DuplicatedFront.eps}
%		\caption{(a) The set $\mathcal{P}$ at $t\approx 0.7$. The fronts arising from the shock are clearly visible. (b) Duplication of the front: The dark blue points are in Accepted, the lighter red points are in the Narrow Band. The black point $p_{\alpha\beta}$ has just been accepted and $F(p_{\alpha \beta})<0$. According to the procedure Update Pile, the circled point underneath it is put in Pile, even though it should not be traversed by the curve when $F<0$.}
%	\label{fig:Ex6NarrowBand}
%	\end{figure} 

On a much more general note, the gluing mechanism between the two formalisms heavily relies on an accurate computation of the normal. In practice, we found that the algorithm is rather sensitive to the accuracy of this quantity. 
Another weakness of the method is that, as it stands, Algorithm \ref{AlgoSideways} may fail when it is not supposed to. i.e., Even though $(x_{i},y_{j})$ or $(x_{\alpha},y_{\beta})$ belongs to $\CC_{t}$ for some $t \in (0,T)$, the algorithm does not assign any value to either of those coordinates. 
Two situations make such a scenario possible: (1) the time steps taken are too small, or (2) too little information obtained from interpolation is available. Recall from Proposition \ref{claim:convergence} that the CFL condition prevents large $\Delta t$. Case (2) can occur if $s\in \NN$, the number of points in the local grid in Algorithm \ref{AlgoSideways} is too small. However, if $s$ is large, Algorithm \ref{AlgoSideways} may not be able to carry out the step outlined in \texttt{line 13}. This happens if the points in NeighSide$(p_{\alpha\beta})$ sample more than one connected component of the set $\{ p\in \MM : \pi_{s}(p) \in [x_{i-s},x_{i+s}] \times [y_{j-s},y_{j+s}] \}$. See Figure \ref{fig:sTooLarge} for an illustration. However, choosing $s$ systematically so as to prevent this situation seems difficult. Ultimately $h$ and $s$ depend on measurable quantities such as the Lipschitz constants of $F$ and its derivatives, as well as the local curvatures of $\Gamma_{t}$. Nevertheless the way those parameters are intertwined and should be chosen is a question we wish to address in future work.

%\begin{figure}
%		\centering
%			%\includegraphics[height = 6cm, width = 8.76cm]{Lemon.png}
%		\caption{Two situations where Algorithm \ref{AlgoSideways} fails.}
%	\label{fig:Algo6Fails}
%	\end{figure} 

\begin{figure}
		\centering
			\includegraphics[width=.7\textwidth]{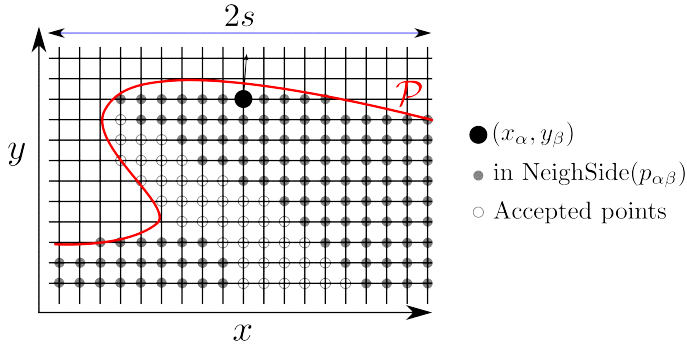}
		\caption{Illustration of what happens if $s$ is chosen too large. Data need to be converted to the $xt$-representation, but the set NeighSide$(p_{\alpha\beta})$ of neighbours of the black point $(x_{\alpha},y_{\beta})$ consists of two connected components.}
	\label{fig:sTooLarge}
	\end{figure} 

The fact that our method is a rather mild modification of the standard FMM has obvious benefits. As featured in all the examples, the sideways representations need only be used to compute a relatively small number of points sampling $\MM$. This allows us to safely predict that the computational complexity of the algorithm is lower than that of pre-existing algorithms used to tackle this problem, such as the LSM or the GFMM. Nonetheless, it is hard at this point to make more precise complexity statements.

%%%%%%%%%%%%%%%%%%%%%%%%%%%%%%%%%%%
%%%%%%%%%%%%%%%%%%%%%%%%%%%%%%%%%%%
%%%%%%%%%%%%%%%%%%%%%%%%%%%%%%%%%%%

\section{Conclusion} 
\label{sec:Conclusions}

Our aim was to devise an algorithm with low complexity able to describe the non-linear evolution of codimension one manifolds subject to a space- \& time-dependent speed function that changes sign. To this end, we illustrated how pre-existing methods can be combined to achieve this goal. The fact that we always dealt with explicit representations of the manifold implied that the dimensionality of the problem was never raised. The resulting algorithm was found to have a global truncation error of $\mathcal{O}(h)$. We tested it against a number of examples, some of which cannot be found in the current literature. 

The algorithm is found to be robust and accurate in all the tests presented. Regarding the complexity of the method, a legitimate concern is to clearly quantify how the success rate of Algorithm \ref{AlgoSideways} depends on the various parameters involved, as well as the speed function $F$ and the manifold $\MM$. Once this is done, more precise statements about the runtime of the algorithm can be made and tested.

%Furthermore we believe that the framework we propose can be modified and improved in many ways. 
%Some disadvantages of the method as it stands are solely due to our choice of implementation. In particular, the way we handle and store the data (through the use of lists and functions) was motivated by clarity of exposition, and is surely not optimal. As noticed in \S \ref{subsubsec:AlgoGetSideways} using only the skewed representation may offer some benefits, such as avoiding to resort
%% we might be better off only working with the skewed representation, to avoid resorting 
%to several attempts in Algorithm \ref{AlgoSideways}. %The works of \textbf{ \cite{NoFreeLunch} or \cite{SmallBalls}} provide alternatives to computing the normal accurately at each point.

Overall, the present work thoroughly introduces a new algorithm, along with proofs of convergence and stability, as well as sturdy numerical results. We believe that the main idea on which it relies -- \ie to change representation based on the speed function $F$ -- may be extended and improved in many ways that shall be explored. 
%It would be of interest to study how the assumptions we have made on $F$ can be relaxed. In turn, this would allow for numerous applications, such as \textbf{ optimal control, ADD APPLICATIONS } \\

%%%%%%%%%%%%%%%%%%%%%%%%%%%%%%%%%%%%%%%%%%%%%%%%
%%%%%%%%%%%%%%%%%%%%%%%%%%%%%%%%%%%%%%%%%%%%%%%%
%%%%%%%%%%%%%%%%%%%%%%%%%%%%%%%%%%%%%%%%%%%%%%%%

\appendix

%%%%%%%%%%%%%%%%%%%%%%%%%%%%%%%%%%%%%%%%%%%%%%%%
%%%%%%%%%%%%%%%%%%%%%%%%%%%%%%%%%%%%%%%%%%%%%%%%
%%%%%%%%%%%%%%%%%%%%%%%%%%%%%%%%%%%%%%%%%%%%%%%%

\section[Quartic involved in the $t$-FMM]{A direct method to compute $\psi_{\mathrm{II}}$ in the $t$-FMM, in 2D}
\label{app:tFMM}

We provide a direct method for solving the minimization problem appearing in Equation (\ref{eq:VladMinimization}), in two dimensions. Introducing $\tau(y)=\frac{h}{|F(\xvec_{ij},\psi(y))|}$, we first use linear interpolation to simplify the quantity we wish to minimize: 
\begin{eqnarray}
&~& \psi(\tilde{\mathbf{x}}) + \sqrt{\xi^{2}+(1-\xi)^{2}}~ \frac{~h}{|F(\xvec_{ij},\psi(\tilde{\mathbf{x}}))|} 
~=~ \psi(\tilde{\mathbf{x}}) + \sqrt{\xi^{2}+(1-\xi)^{2}}~ \tau(\tilde{\mathbf{x}}) \nonumber \\
&\approx& \xi \psi(\xvec_{i-1,j})+(1-\xi) \psi(\xvec_{i,j+1}) + \sqrt{\xi^{2}+(1-\xi)^{2}}~ \left( \xi \tau(\xvec_{i-1,j})+(1-\xi) \tau(\xvec_{i,j+1}) \right) \nonumber \\
&=:& f(\xi)
\end{eqnarray}
Minimizing $f$ over $\xi \in (0,1)$ amounts to finding the roots of $0 = c_{4}\lambda^{4} + c_{3}\lambda^{3} + c_{2}\lambda^{2} + c_{1}\lambda + c_{0}$ where $\lambda \in (0,1)$ is such that $f'(\lambda)=0$. % and the coefficients are given as:
%\begin{eqnarray*}
%\begin{array}{lll}
%c_{4}=b_{4} 							& b_{4}= a^{2}_{2}  				& ~ \\ 
%c_{3}=b_{3} 							& b_{3}= 2a_{1}a_{2} 				& ~ \\ 
%c_{2}=b_{2}-2(\psi(x_{i-1,j})-\psi(x_{i,j+1}))^{2} 	& b_{2}= 2a_{0}a_{2}+a^{2}_{1} 	& a_{2} = 4(\tau(x_{i-1,j})-\tau(x_{i,j+1}))\\ 
%c_{1}=b_{1}+2(\psi(x_{i-1,j})-\psi(x_{i,j+1}))^{2} 	& b_{1}= 2a_{0}a_{1} 				& a_{1} = 5\tau(x_{i,j+1})-3\tau(x_{i-1,j}) \\ 
%c_{0}=b_{0}-2(\psi(x_{i-1,j})-\psi(x_{i,j+1}))^{2} 	& b_{0}= a^{2}_{0} 				& a_{0} = \tau(x_{i-1,j})-2\tau(x_{i,j+1}) 
%\end{array}
%\end{eqnarray*}
This quartic can be solved either directly with closed formulas, or with Newton's method --- we use the latter. For each root $r_{i} \in (0,1)$ the corresponding value of $\psi$ is computed as $\psi_{\mathrm{II},r_{i}}=f(r_{i})$. If $\psi_{\mathrm{II},r_{i}}< \psi(\xvec_{i-1,j})$ or $\psi_{\mathrm{II},r_{i}}<\psi(\xvec_{i,j+1})$, then $\psi_{\mathrm{II},r_{i}}$ is discarded. Values arising from minimization in one dimension are also computed as $\psi_{\mathrm{II},0} = \psi(\xvec_{i,j+1}) + \tau(\xvec_{i,j+1})$ and $\psi_{\mathrm{II},1} = \psi(\xvec_{i-1,j}) + \tau(\xvec_{i-1,j})$.
The global minimum is found by comparing all those values.

\section{Algorithm \ref{AlgoFMM}, standard FMM} 
\label{subsec:AlgoFMM}
We revisit the standard Fast Marching Method algorithm, using some of the notation we have introduced. 

\begin{algorithm}[h]
\caption{Solve $|\nabla \psi(x,y)| = \frac{1}{|F(x,y)|}$}\label{AlgoFMM}
\begin{algorithmic}
%	\State $u_{-} \gets \pi_{t}(p_{i-1j})$ if $p_{i-1j} \in \mathrm{NeighEik}((x_{i},y_{j}))$, $+\infty$ otherwise. 
%	\State $u_{+} \gets \pi_{t}(p_{i+1j})$ if $p_{i+1j} \in \mathrm{NeighEik}((x_{i},y_{j}))$, $+\infty$ otherwise. 
%	\State $v_{-} \gets \pi_{t}(p_{ij-1})$ if $p_{ij-1} \in \mathrm{NeighEik}((x_{i},y_{j}))$, $+\infty$ otherwise. 
%	\State $v_{+} \gets \pi_{t}(p_{ij+1})$ if $p_{ij+1} \in \mathrm{NeighEik}((x_{i},y_{j}))$, $+\infty$ otherwise. 

	\State $u_{\pm} \gets \pi_{t}(p_{i\pm1j})$ if $p_{i\pm1j} \in \mathrm{NeighEik}((x_{i},y_{j}))$, $+\infty$ otherwise. 
	\State $v_{\pm} \gets \pi_{t}(p_{ij\pm1})$ if $p_{ij\pm1} \in \mathrm{NeighEik}((x_{i},y_{j}))$, $+\infty$ otherwise. 

	\State $u \gets \min(u_{-},u_{+})$, $\qquad$ $v \gets \min(v_{-},v_{+})$, ~

	\If{$\max(u,v)-\min(u,v)<\frac{h}{|F(x_{i},y_{j})|}$}
		\State $\psi_{ij}=\frac{1}{2} \left( (u+v)+\sqrt{2 \left( \frac{h}{F(x_{i},y_{j})}\right)^2-(u-v)^2} \right)$
	\Else
		\State $\psi_{ij} = \min(u,v)+\frac{h}{|F(x_{i},y_{j})|}$
	\EndIf    
\end{algorithmic}
\end{algorithm}

%%%%%%%%%%%%%%%%%%%%%%%%%%%%%%%%%%%%%%%%%%%%%%%%
%%%%%%%%%%%%%%%%%%%%%%%%%%%%%%%%%%%%%%%%%%%%%%%%
%%%%%%%%%%%%%%%%%%%%%%%%%%%%%%%%%%%%%%%%%%%%%%%%

\section{Implementation details for the examples}
\label{app:Details}

\subsection{Solvers used}
\label{subsec:SolversUsed}

We give some details about the examples presented in \S \ref{sec:Accuracy}. All tests were performed using \textsc{Matlab}$^\circledR$ \cite{Matlab}. In particular, finding the minimum value in the Narrow Band is done using the command \verb min ~.

%\paragraph{Initialization} All data are initialized with exact values.

%\paragraph{Interpolation} The interpolation scheme that converts data from one representation to another is one-dimensional linear interpolation, as illustrated on Figure \ref{fig:ConversionOfData}. This procedure is $\mathcal{O}(h)$.

%\paragraph{The normals $\nthree$ and $\ntwo$} 
Whenever a value $\psi_{ij}$ is computed by the $(t)$-FMM, the normal $\nthree_{ij}$ is approximated using the one-sided derivatives involving the points used in the computation of $\psi_{ij}$. For example: if two-dimensional optimization was used in Quadrant III to obtain $\psi_{ij}$, then 
\begin{eqnarray}
\vec{v} = \left( \frac{\psi_{ij}-\psi_{i-1,j}}{h}, \frac{\psi_{ij}-\psi_{i,j-1}}{h}, -\mathrm{\Orientth}(p_{ij}) \right) 
\quad \mathrm{and} \quad 
\nthree(p_{ij}) = \frac{\vec{v}}{|\vec{v}|} 
\end{eqnarray}
%In the instance where one-dimensional optimization is used, the lack of information in one of the directions can be problematic when deciding whether to include neighbouring points in \Pp ~(\cf \texttt{line 12} in Algorithm \ref{AlgoMainLoop}). Taking an average of the normals at the points in NeighEik$((x_{i},y_{j}))$ provides an approximation which is accurate enough for our purposes. 
Within Algorithm \ref{AlgoSideways}, we approximate the normal as follows. For clarity, say the points $p^{k}_{j} = (\psi^{k}_{j},y_{j},t^{k})$ and $p^{k-1}_{j} = (\psi^{k-1}_{j},y_{j},t^{k-1})$ computed in the $yt$-representation with $x$-orientation $a$ were used to obtain $p_{ij} = (x_{i},y_{j},\psi_{ij})$. Then 
\begin{eqnarray}
\vec{v} = \left( -a, a ~\frac{\psi^{k-1}_{j+1}-\psi^{k-1}_{j-1}}{2h}, a ~ \frac{\psi^{k}_{j}-\psi^{k-1}_{j}}{dt}  \right) 
\quad \mathrm{and} \quad 
\nthree(p_{ij}) = \frac{\vec{v}}{|\vec{v}|} 
\end{eqnarray}
Note that this is not an approximation of the true normal at $p_{ij}$, which is $( -a, -\phi_{x} \psi_{y},$ $ -\phi_{x} \psi_{t} )$. However, the only two salient information we need from $\nthree$ are: the sign of $\nthree_{3}$ and the direction of $\ntwo$. 
The two-dimensional normal is simply obtained from $\nthree$ as $\ntwo = \frac{(\nthree_{1},\nthree_{2})}{|(\nthree_{1},\nthree_{2})|}$.

%\paragraph{The Sign Test} We check the number of times the speed changes sign along a line sampled with 100 points. 

%\paragraph{Sorting} Finding the minimum value in the Narrow Band is done using the \verb min ~command in \textsc{Matlab}. %{\red According to \cite{UsersGuideMatlab}, this uses a quicksort algorithm.} 

%\paragraph{Reconstructing the curve} {\red We use the Fast Surface Reconstruction method \cite{Zhao} to get a continuous approximation of $\MM$ from the set of points \Acc. \textbf{Actually, I don't.}} 

\subsection{Choice of parameters} In all examples, the number of points in each dimension is $N+1$, and the spatial grid spacings are even: $h=dx=dy$. The size of the local grid in Algorithm \ref{AlgoSideways} was set to be $s= \lfloor \frac{N}{3} \rfloor$. As discussed in \S \ref{subsubsec:AlgoMainLoopSideways}, we use adaptive time-stepping, in those examples where $F$ depends on time. In the fine part, before the time where $F=0$, we set $\Delta t = r_{1} h$. Passed that time, we let $\Delta t = r_{2} h$. 
To assess the convergence of the sideways methods, a $yt$-grid with spacings $h$ and $\Delta t = h/2$ was built.%, and $\psi(y,t)$ was initialized for some $t_{0}$ with exact values. Then Algorithm \ref{AlgoSidewaysPDE} was run for each time step. The error was measured on all those values such that $\psi(y_{m},t^{n})<10^6$.% using Method 1 for Examples 1-4, and for the bottom part of Example 5, and Method 2 for the top part of example 5. 
Remark that the exact normal $\nthree$ was assigned to the points as they were accepted in all the examples, except Example 1 where it was computed as explained in \S \ref{subsec:SolversUsed}.

%%%%%%%%%%%%%%%%%%%%%%%%%%%%%%%%%%%%%%%%%%%%%%%%
%%%%%%%%%%%%%%%%%%%%%%%%%%%%%%%%%%%%%%%%%%%%%%%%

\subsection{Example 1}

The exact solution to the Level-Set Equation is $\phi(x,y,t) = \sqrt{x^{2}+y^{2}} -R(t)$ where $R(t) = \left( r_{0} - \frac{e^{10t}-1}{10e}+t \right)$. 
Domain: $[-.321,.319]^{2}$. $T_{F}=0.3$. $xt$- and $yt$-rep.: $r_{1} = 1/3$, $r_{2} = 2$. Skewed rep.: $r_{1} = r_{2} = 1$.
Domain for convergence of Algo. \ref{AlgoSidewaysPDE}: $(y,t) \in [-0.25,0.25]\times [0,0.3]$. 

%%%%%%%%%%%%%%%%%%%%%%%%%%%%%%%%%%%%%%%%%%%%%%%%
%%%%%%%%%%%%%%%%%%%%%%%%%%%%%%%%%%%%%%%%%%%%%%%%

%\subsection{Example 2}
%
%The speed is $F = \alpha \sin \left(\beta (t+\gamma) \right)$ with $\alpha = 0.5$, $\beta=4$ and $\gamma=0.3$. The exact solution to the Level-Set Equation $\phi(x,y,t) = \sqrt{x^{2}+y^{2}} -R(t)$ where $R(t) = \left( r_{0} + \frac{\alpha}{\beta} \left[ \cos\left( \beta \gamma \right) - \cos\left( \beta ( t+\gamma) \right) \right] \right) $.
%Domain: $[-.61,0.59]^2$. $T_{F}=5$. $xt$- and $yt$-rep.: $r_{1} = 1/2$, $r_{2} = 5$. Skewed rep.: $r_{1} = 1$, $r_{2} = 10$.
%Domain for convergence of Algo. \ref{AlgoSidewaysPDE}: $(y,t) \in [-0.25,0.25]\times [0,5]$. 

%%%%%%%%%%%%%%%%%%%%%%%%%%%%%%%%%%%%%%%%%%%%%%%%
%%%%%%%%%%%%%%%%%%%%%%%%%%%%%%%%%%%%%%%%%%%%%%%%

\subsection{Example 2}

The signed distance function to the curve $\CC_{t}$ is given as $\phi(x,y,t)=\sqrt{(x-x_{c}(t))^{2}+y^{2}}-r(t)$ where $x_{c}(t) = r_{0}\sinh t$ and $r(t) = r_{0}\cosh t$. Note that $\phi$ does not solve the Level-Set Equation.
Domain: $[-1.01,0.99]^2$. $T_{F}=1$. $xt$- and $yt$-rep.: $r_{1} = 1/3$, $r_{2} = 2$. Skewed rep.: $r_{1} = 1/3$, $r_{2} = 5$.
Domain for convergence of Algo. \ref{AlgoSidewaysPDE}: $(y,t) \in [-0.25,0.25]\times [0,1]$ and $(x,t) \in [-0.25,0.25]\times [0,1]$.

%%%%%%%%%%%%%%%%%%%%%%%%%%%%%%%%%%%%%%%%%%%%%%%%
%%%%%%%%%%%%%%%%%%%%%%%%%%%%%%%%%%%%%%%%%%%%%%%%

\subsection{Example 3}
\label{app:Example4}

The exact solution to the Level-Set Equation is $\phi(x,y,t) = \sqrt{(x-g t)^2+y^2}-\left( r_{0}+c t \right)$
where $b=10$, $c=1/2$ and $g(t) = \arctan \left(b(t-0.5)\right) + \frac{\pi}{2}$. The speed is 
\begin{eqnarray}
F = \frac{(x-gt)(g't+g)}{\sqrt{(x-gt)^{2}+y^{2}}} + c
~
\Longrightarrow
~
F \approx
\left\{ \begin{array}{ll}
c & \mathrm{for~} t \mathrm{~small} \\ 
\frac{(x-\pi t)\pi}{\sqrt{(x-\pi t)^{2}+y^{2}}} + c  & \mathrm{for~} t \mathrm{~large} \\ 
\end{array} \right.
\end{eqnarray} 
We expect the circle to first expand (when $t$ is small), and then expand while moving to the right with speed $\pi$ (when $t$ is large). 
Domain: $[-1.51,+1.49]^{2}$. $T_{F}=0.5$. $xt$- and $yt$-rep.: $r_{1} = 1/3$, $r_{2} = 2$. Skewed rep.: $r_{1} = 1/3$, $r_{2} = 5$.
Domain for convergence of Algo. \ref{AlgoSidewaysPDE}: $(y,t) \in [-0.25,0.25]\times [0,0.5]$.

%%%%%%%%%%%%%%%%%%%%%%%%%%%%%%%%%%%%%%%%%%%%%%%%
%%%%%%%%%%%%%%%%%%%%%%%%%%%%%%%%%%%%%%%%%%%%%%%%

\subsection{Example 4}

The set $\CC_{0}$ consists of two disjoint circles of radius $r_{0}=0.25$, with centres at $(-0.3,0)$ and $(0.3,0)$. 
The speed is $F = 1-e^{2t-1}$. The circles touch along the $y$-axis when $t\approx 0.08$. When $t<0.5$ the exact solution to the Level-Set Equation is $\phi(x,y,t) = \min \left\{ \sqrt{(x+0.3)^{2}+y^{2}} - R(t), \sqrt{(x-0.3)^{2}+y^{2}} - R(t)  \right\} $
where $R(t)= r_{0} - \frac{e^{2t}-1}{2e}+t $. 
Domain: $[-1.5+ 0.01e,+1.5+ 0.01e]^{2}$. $T_{F}=1.2$. $xt$- and $yt$-rep.: $r_{1} = 1/3$, $r_{2} = 2$. Skewed rep.: $r_{1} = 1/3$, $r_{2} = 5$.
Domain for convergence of Algo. \ref{AlgoSidewaysPDE}: $(y,t) \in [-0.5,0.5]\times [0.2,0.5]$ and $(y,t) \in [-0.5,0.5]\times [0.5,.52]$.

%%%%%%%%%%%%%%%%%%%%%%%%%%%%%%%%%%%%%%%%%%%%%%%%
%%%%%%%%%%%%%%%%%%%%%%%%%%%%%%%%%%%%%%%%%%%%%%%%

\subsection{The Almond example} 
\label{subsec:AlmondExample}

The exact solution to the Level-Set Equation is 
\begin{eqnarray}
\phi(x,y,t) &=& \left( \sqrt{x^{2}+y^{2}}-r_{0}+\frac{e^{ct}-1}{ce}-t(1+C)  \right)+ \frac{t|xt-y|}{\sqrt{1+t^{2}}} \\
&=:& \tilde{\phi}(x,y,t) + g(x,y,t)
\end{eqnarray}
The constants are set to be: $r_0 = 1/4$, $c = 1$, and $C = .65$. The function $\phi$ is made up of two parts: $\tilde{\phi}$ is qualitatively the same as in Example 1. %The effect of $g$ is to create a shock along the line $y=xt$. The shock appears at time $t=0$, and becomes 'sharper' as time advances. 
Domain: $[-0.5,0.5]^2$. $T_{F}=1.9$. $xt$- and $yt$-rep.: $r_{1} = 1/3$, $r_{2} =2$. Skewed rep.: $r_{1} = 1/2$, $r_{2} = 6$. \\

%%%%%%%%%%%%%%%%%%%%%%%%%%%%%%%%%%%%%%%%%%%%%%%%
%%%%%%%%%%%%%%%%%%%%%%%%%%%%%%%%%%%%%%%%%%%%%%%%

\textbf{Acknowledgements}
The authors wish to thank Prof.\mbox{~}A.Oberman for helpful discussions. 
%The first author thanks the members of the Mathematics \& Statistics department at McGill University, in particular Prof.A.Oberman, for helpful discussions. 
The second author would like to thank the organizers of the 2011 BIRS workshop ``Advancing numerical methods for viscosity solutions and applications'', Profs.\mbox{~}Falcone, Ferretti, Mitchell, \& Zhao for stimulating discussions which eventually lead to the present work.

%%%%%%%%%%%%%%%%%%%%%%%%%%%%%%%%%%%%%%%%%%%%%%%%
%%%%%%%%%%%%%%%%%%%%%%%%%%%%%%%%%%%%%%%%%%%%%%%%
%%%%%%%%%%%%%%%%%%%%%%%%%%%%%%%%%%%%%%%%%%%%%%%%
% BIBLIOGRAPHY %

\bibliography{bib}{}
\bibliographystyle{plain}

\end{document}